\documentclass[final]{siamltex}

\usepackage{graphicx}
\usepackage{amsmath}
\usepackage{algorithm}
\usepackage{algorithmicx}
\usepackage[noend]{algpseudocode}
\usepackage{color}
\usepackage{bm}
\usepackage{amsfonts}
\usepackage{float}
\usepackage{subfig}
\usepackage{epstopdf}

\newcommand{\R}{\mathbb{R}}
\newcommand{\tr}{^{\sf T}}
\newcommand{\m}[1]{{\bf{#1}}}

\newcommand{\C}[1]{{\cal {#1}}}

\newtheorem{remark}{Remark}[section]
\newtheorem{assum}{Assumption}
\newtheorem{defin}{Definition}

\title{Decentralized Conjugate Gradient and Memoryless BFGS Methods 
\thanks{
The authors gratefully acknowledge support by the National Science Foundation under grants...
}}
\author{
	    Liping Wang\thanks{{\tt wlpmath@nuaa.edu.cn},
		School of Mathematics,
		Nanjing University of Aeronautics and Astronautics.}
	\and
    Hao Wu \thanks{{\tt wuhoo104@nuaa.edu.cn},
    School of Mathematics,
Nanjing University of Aeronautics and Astronautics.}
\and
Hongchao Zhang\thanks{{\tt hozhang@math.lsu.edu},
Department of Mathematics, 
Louisiana State University}
}
\begin{document}
\maketitle
\begin{abstract}
This paper proposes a new decentralized conjugate gradient (NDCG) method and a decentralized memoryless BFGS (DMBFGS) method for the nonconvex and strongly convex
decentralized optimization problem, respectively, of minimizing a finite sum of continuously differentiable functions over a fixed-connected undirected network. 
Gradient tracking techniques are applied in these two methods to enhance their convergence properties and the numerical stability. 
In particular, we show global convergence of NDCG with constant stepsize for general nonconvex smooth decentralized optimization.
Our new DMBFGS method uses a scaled memoryless BFGS technique and only requires gradient information to approximate second-order information of
the component functions in the objective. We also establish global convergence and linear convergence rate of DMBFGS with constant stepsize for strongly convex smooth decentralized optimization.
Our numerical results show that NDCG and DMBFGS are very efficient in terms of both iteration and communication cost
compared with other state-of-the-art methods for solving smooth decentralized optimization. 
\end{abstract}
\begin{keywords}
decentralized optimization, gradient tracking, conjugate gradient method, constant stepsize, memoryless BFGS, global convergence, linear convergence rate
\end{keywords}

\begin{AMS}
\end{AMS}

\pagestyle{myheadings}
\thispagestyle{plain}
	\section{INTRODUCTION}
In this paper, we consider the following decentralized optimization problem over an undirected and connected network containing $n$ nodes
\begin{equation}\label{obj_fun1}
	\mathop {\min } \limits_{\m{z} \in {\R^p}} \; F(\m{z}):= \frac{1}{n}\sum\limits_{i = 1}^n {{f_i}(\m{z})},
\end{equation}
where the local objective function $f_i:\R^{p} \rightarrow \R$ is  continuously differentiable. 
Consider the underlying network $\mathcal{G}=\left(\mathcal{V},\mathcal{E}\right)$, where $\mathcal{V} = \{1,\ldots,n\}$ is the set of nodes, and $\mathcal{E}$ is the collection of unordered edges.
We denote two nodes as neighbors if they are connected by an
edge. In decentralized setting, there does not exist one central server to gather local information from other nodes, compute shared global information, or broadcast it to all other nodes. Each local function $f_i$ is only known to node $i$ and all the nodes collaborate with their neighbors through information exchange (i.e., communication) to obtain the consensus minimizer. Decentralized optimization has wide applications including decentralized resources control \cite{fusco2021decentralized}, wireless networks \cite{jeong2022asynchronous}, decentralized machine learning \cite{zhang2022distributed}, power systems \cite{chen2020fully}, federated learning \cite{pillutla2022robust}.

Because of wide and important practical applications, decentralized optimization has been extensively studied, where gradient-based first-order methods have attracted much attention due to their simple implementation
and low computation cost at each iteration. Among numerous first-order methods, decentralized gradient descent (DGD) \cite{nedic2009distributed,yuan2016convergence,zeng2018nonconvex} methods 
were most early developed and widely used.  However, DGD only converges to a stationary point of the original problem with a diminishing stepsize.
When using a constant stepsize, DGD often converges to a stationary point of a Lyapunov function instead of the original problem itself \cite{yuan2016convergence}.
Gradient tracking (GT) techniques \cite{xu2015augmented,qu2017harnessing,nedic2017achieving,nedic2017geometrically,xin2019distributed, gao2022achieving, zhang2020distributed,song2024optimal}
were developed for using constant stepsize without losing the exact convergence, 
which utilize a dynamical average consensus \cite{zhu2010discrete}  for local gradient approximations at any nodes to track the global average gradient.
Actually, GT-based methods were originally studied for convex and strongly convex optimization problems \cite{xu2015augmented,qu2017harnessing,nedic2017achieving,nedic2017geometrically}.  
Lu et al. \cite{lu2019gnsd} first propose a GT-based nonconvex stochastic decentralized method, which extends the deterministic GT techniques \cite{qu2017harnessing,nedic2017achieving} to the nonconvex stochastic setting.
Hong et al. \cite{hong2022divergence} relax the Lipschitz condition on the local gradients and propose a multi-stage GT method with global convergence. 
Takezawa et al. \cite{takezawa2022momentum} further propose a momentum tracking (MT) method by introducing and dynamically averaging momentum into stochastic GT to accelerate the convergence.
The MT technique as well as the loopless Chebyshev acceleration \cite{song2024optimal} are also used in  \cite{huang2024accelerated}. 
Considering GT-based and MT-based methods require two rounds of communication per iteration, Aketi et al. \cite{NEURIPS2023_98f8c89a} propose to apply
the tracking mechanism for the variable updates to save communication, while slow convergence is often expected for the method because of the dissimilarity between local and global gradients.

Because of the great advantages of nonlinear conjugate gradient (CG) methods over standard gradient descent methods for smooth nonconvex nonlinear optimization,    
decentralized CG methods have also been recently studied for decentralized optimization. 
However, the development of decentralized CG methods is very limited and far from mature.
Xu et al. \cite{xu2020distributed} propose a distributed online CG algorithm for a distributed online optimization,
where uniformly bounded conjugate parameters, diminishing stepsizes and  extra projection step are needed for global convergence.
A decentralized Riemannian CG descent Method is proposed in  \cite{chen2023decentralized}, which requires the stepsizes to decrease and satisfy the local strong Wolfe condition on each node. 
By using the dynamic average consensus technique\cite{zhu2010discrete}, a decentralized CG method using constant stepsize is developed \cite{shorinwa2023distributed}.
However, we think the convergence analysis in \cite{xu2020distributed,chen2023decentralized,shorinwa2023distributed} is inadequate or problematic, which will be more explained in Section 2.1.


When the problem is convex or strongly convex, there is more flexibility to design the algorithms.
 Xin and Khan \cite{xin2019distributed} present a distributed heavy ball method, called ABm, that combines GT with an alternative momentum. 
 Gao et al. \cite{gao2022achieving} introduce the Barzilai-Borwein (BB) \cite{barzilai1988two} technique into the adapt-then-combine (ATC) GT method \cite{xu2015augmented,nedic2017geometrically} to adaptively compute the  stepsizes for each node. 
Optimal GT (OGT) \cite{song2024optimal} is the first decentralized gradient-based method, not relying on inner loops to reach the optimal complexities for minimizing smooth strongly convex optimization problems. 
Quasi-Newton techniques are more widely and conveniently used in strongly convex optimization, 
since the curvature information of the Hessian of the objective function can be easily captured with a low computation cost.
 Eisen et al. \cite{eisen2017decentralized} propose a decentralized BFGS (DBFGS) method to solve a penalized problem, where the consensus constraint violation is reduced by a penalty approach.
However, the DBFGS \cite{eisen2017decentralized} is an inexact penalty method in the sense that the penalty parameter needs to go to infinity for ensuring global convergence. 
PD-QN \cite{eisen2019primal} improves DBFGS in the primal-dual framework. A decentralized ADMM \cite{li2021bfgs}  incorporates the BFGS quasi-Newton technique to improve computation efficiency.
 Note that quasi-Newton methods overcome the difficulty of computing Hessian matrices but can not avoid the high memory requirement of the Hessian approximation matrices. 
Zhang et al. \cite{zhang2023variance} propose a damped limited-memory BFGS (D-LM-BFGS) method and a damped regularized limited-memory DFP (DR-LM-DFP) method, for which only modest memory is needed. 
In addition, D-LM-BFGS reduces computation cost by realizing a two-loop recursion. 
Existing decentralized quasi-Newton methods usually add a regularization term or take a damping technique to ensure 
the generated quasi-Newton matrices are positive definite and have bounded eigenvalues. 

In this paper, new decentralized CG and memoryless quasi-Newton methods are proposed for solving nonconvex and strongly convex optimization problems, respectively. 
Our main contributions are as follows. 
\begin{itemize}
	\item [1.] Motivated by the centralized Polak-Ribi{\`e}re-Polyak CG method that can guarantee convergence using constant stepsizes for minimizing smooth nonconvex optimization, 
we propose a new decentralized CG method (NDCG). To the best of our knowledge, NDCG is the first decentralized CG method using constant stepsize with global convergence for
smooth nonconvex minimization problems. The convergence of NDCG only requires Lipschitz continuity of the gradients and lower boundedness of the objective function, which is the 
same as those requirements for the centralized nonlinear CG method.
	Our numerical experiments show NDCG is very effective compared with other decentralized gradient-based methods for nonconvex optimization.
	\item[2.] Considering the relationship between the nonlinear CG methods and the memoryless quasi-Newton methods, 
we develop a decentralized memoryless BFGS method (DMBFGS) where the updating direction is derived from the scaled memoryless BFGS technique \cite{shanno1978convergence}.
The DMBFGS has property that only gradient information is used to capture the second-order information. In addition,
the quasi-Newton matrix generated by DMBFGS is positive definite and has bounded eigenvalues without using regularization or damping techniques. 
       We show DMBFGS  has linear convergence rate when the objective function is strongly convex.
       Our numerical experiments also show DMBFGS performs better than other advanced decentralized first-order methods for minimizing strongly convex optimization.
\end{itemize}
The paper is organized as follows. In Section~2, we first review the previously developed decentralized CG methods and then propose our NDCG and DMBFGS methods.
Global convergence properties of the two methods are also studied in this section. 
Numerical experiments of comparing our new methods with other well-established first-order methods for solving decentralized optimization are presented in Section~3.
We finally draw some conclusions in Section 4.

\subsection{Notation}
We use uppercase boldface letters, e.g.~$\m{W}$, for matrices and lowercase boldface letters,  e.g.~$\m{w}$, for vectors.
For any vectors $\m{v}_i \in \R^p$, $i=1,\ldots,n$, we define  $\bar{\m{v}}=\frac{1}{n} \sum_{i=1}^n \m{v}_i$ and $\m{v} = [\m{v}_1; \m{v}_2; \ldots, \m{v}_n] \in \R^{np}$.
Given an undirected network  $\mathcal{G}=\left(\mathcal{V},\mathcal{E}\right)$,
 let $\m{x}_i$ denote the local copy of the global variable $\m{z}$ at node $i$ and $\mathcal{N}_i$ denote the set consisting of the neighbors of node $i$ 
(for convenience, we treat node $i$ itself as one of its neighbors). 
We define $f(\m{x}) = \sum_{i = 1}^n {{f_i}({\m{x}_i})}$ and use $\m{g}^t$, $\m{g}_i^t$ to stand for $\nabla f(\m{x}^t)$, $\nabla f_i(\m{x}^t_{i})$ respectively,
where, for clarification, the gradient of $f(\m{x})$ is defined as
$ \nabla f(\m{x}) =\left[ \nabla f_1(\m{x}_{1}); \nabla f_2(\m{x}_{2}); \ldots, \nabla f_n(\m{x}_{n}) \right] \in \R^{np}$. 
In addition,  we define $\overline{\nabla} f(\m{x}^t)=\frac{1}{n} \sum_{i=1}^n \nabla f_i(\m{x}_i^t) 
\in \R^p$.
We say that $\m{x}$ is consensual or gets consensus if ${\m{x}_1}={\m{x}_2}=\ldots={\m{x}_n}$. 
$\m{I}_p$ denotes the $p \times p$ identity matrix and $\m{I}$ denote $\m{I}_{np}$ for simplicity. Kronecker Product is denoted as $\otimes$.
Given a vector $\m{v}$ and a matrix $\m{N}$, $\operatorname{span}(\m{v})$ stands for the linear subspace spanned by $\m{v}$;  $\left\|\m{v}\right\|_{\m{N}}^2$ denotes $\m{v}\tr \m{N} \m{v}$;
$\mbox{Null}(\m{N})$ and $\m{N}\tr$  denote the null space and transpose of $\m{N}$, respectively;
${\lambda _{\min }(\m{N})}$,  ${\lambda _{\max }(\m{N})}$, and $\rho(\m{N})$ denote smallest eigenvalue, largest eigenvalue, and  the spectral radius of $\m{N}$,  respectively;
For matrices $\m{N}_1$ and $\m{N}_2$ with same dimension, $\m{N}_1 \succeq \m{N}_2$ means $\m{N}_1 - \m{N}_2$ is positive semidefinite, while $\m{N}_1 \ge \m{N}_2$ means $\m{N}_1 - \m{N}_2$ is component-wise nonnegative.
We denote $\operatorname{log}_{10}(\cdot)$ by $\operatorname{log}(\cdot)$  and define $\m{M}=\frac{1}{n} \m{1}_n\m{1}_n\tr\otimes \m{I}_p$ where $\m{1}_n \in \R^n$ denotes the
vector with all components ones. 

\section{ALGORITHM DEVELOPMENT AND CONVERGENCE RESULTS}
In this section, we derive our new decentralized CG  and memoryless BFGS methods for solving nonconvex and strongly convex 
optimization problems, respectively. 
To analyze global convergence of the new methods, we need the following assumption and some preliminary results.

\begin{assum}\label{as0}
The local objective functions $\{f_i \}_{i=1}^n$ are bounded below and Lipschitz continuously differentiable, that is, for any $\m{z}, \tilde{\m{z}} \in \R^p$ and $i=1, \ldots,n$, 
we have  $f_i(z)>-\infty$ and 
	\begin{equation}\label{3.1}
		\left\| {\nabla {f_i}(\m{z}) - \nabla {f_i}(\tilde{\m{z}})} \right\| \le L \left\| {\m{z} - \tilde{\m{z}}} \right\|,
	\end{equation}
\end{assum}
where $L >0$ is a Lipschitz constant.

In decentralized optimization it is convenient to parameterize communication by a mixing matrix $\tilde{\m{W}}=[ \tilde{W}_{i j} ]\in \R^{n \times n}$, which is defined as follows.
\begin{defin}\label{mix}
	(Mixing matrix $\tilde{\m{W}}$ for given network $\mathcal{G} = \left(\mathcal{V},\mathcal{E}\right)$)
	\begin{itemize}
		\item [1.] $\tilde{\m{W}}$ is nonnegative, where each component $\tilde{W}_{ij}$ characterizes the active link $(i,j)$, i.e., $\tilde{{W}}_{i,j}>0$ if $j \in \mathcal{N}_i$;  $\tilde{{W}}_{i,j}=0$, otherwise.
		\item [2.] $\tilde{\m{W}}$ is symmetric and doubly stochastic, i.e., $\tilde{\m{W}}=\tilde{\m{W}}\tr$ and $\tilde{\m{W}}\m{1}_n=\m{1}_n$.\\
	\end{itemize}
\end{defin}
There are a few common choices for the mixing matrix $\tilde{\m{W}}$, such as the Laplacian-based constant edge weight matrix \cite{sayed2014diffusion} and
the Metropolis constant edge weight matrix \cite{xiao2007distributed}.
Let $\lambda_{i}(\tilde{\m{W}})$ denote the $i$-th largest eigenvalue of $\tilde{\m{W}}$ and $\sigma$ be the second largest magnitude eigenvalue of $\tilde{\m{W}}$. 
Then, the following properties hold \cite{xin2019distributed}.
\begin{lemma}\label{property W}
For $\tilde{\m{W}}$ defined in Definition~\ref{mix} and  $\m{W} :=\tilde{\m{W}}\otimes\m{I}_p$, we have
	\begin{itemize}
		\item [1.] $1=\lambda_{1}(\tilde{\m{W}})>\lambda_{2}(\tilde{\m{W}})\geq\ldots\geq\lambda_{n}(\tilde{\m{W}})>-1$;
		\item [2.] $	0<\rho(\m{W}-\m{M})=\sigma=\max \left\{|\lambda_{2}(\tilde{\m{W}})|, |\lambda_{n}(\tilde{\m{W}})|\right\}<1$;
		\item [3.] $\m{M}=\m{M}\m{W}=\m{W}\m{M}$;
		\item [4.]  $\|\m{W}\m{x}-\m{M}\m{x}\|=\|(\m{W}-\m{M})(\m{x}-\m{M}\m{x})\| \leq \sigma \|\m{x}-\m{M}\m{x}\|$
		for any $\m{x} \in \R^{np}$. 
	\end{itemize}
\end{lemma}

We also introduce some well-known results used in our convergence analysis.
\begin{lemma}\label{young}
	(Young's inequality) For any two vectors $\m{v}_1,\m{v}_2\in \R^p$, $\eta>0$, we have
\[
2\m{v}_1\tr\m{v}_2 \leq \eta \|\m{v}_1\|^2+ \frac{1}{\eta}\|\m{v}_2\|^2 \quad \mbox{and} \quad
		\|\m{v}_1+\m{v}_2\|^2 \leq (1+\eta) \|\m{v}_1\|^2+ \left(1+\frac{1}{\eta}\right)\|\m{v}_2\|^2.
\]
\end{lemma}
\begin{lemma}\label{important}
	(Jensen's inequality) For any set of vectors $\{\m{v}_i\}_{i=1}^n \subset \R^p$, we have
	$$\left\|\frac{1}{n}\sum_{i=1}^n\m{v}_i\right\|^2\leq\frac{1}{n}\sum_{i=1}^n\left\|\m{v}_i\right\|^2.$$ 
\end{lemma}
\begin{lemma}\label{important1}
	\cite[Theorem 8.3.2.]{horn2012matrix}\quad Let $\m{M} \in  \R^{p \times p}$ be a nonnegative matrix and $\m{m} \in \R^p$ be a positive vector. 
If $\m{M}\m{m} \leq \omega \m{m}$ for some $\omega \ge 0$, we have $\rho(\m{M})\leq\omega$.
\end{lemma}

\subsection{Some Literature Review On Decentralized CG}
Let us consider a simple decentralized CG (SDCG) method which was originally mentioned in \cite{shorinwa2023distributed}. 
We notice that the methods from \cite{xu2020distributed,chen2023decentralized,shorinwa2023distributed} are actually developed based on SDCG. 
The iterative formula of SDCG with respect to node $i$ is given as 
\begin{align} 
	& \m{x}_i^{t+1}=\sum_{j \in \mathcal{N}_i} \tilde{W}_{i j} \m{x}_j^{t}+\alpha \m{d}_i^{t}, \label{dcg1}\\ 
	& \m{d}_i^{t+1}=-\m{g}_i^{t+1}+\beta_i^{t+1} \m{d}_i^{t} \label{dcg2}
\end{align}
with initialization $\m{d}_i^{0}=-\m{g}_i^{0}$, where $\alpha>0$ is the stepsize and  $\beta^t_{i}$ is called the conjugate parameter. 
Here, we consider SDCG using constant stepsize. 
 The following is a list of several well-known formulas of $\beta^t_{i}$ for different nonlinear CG methods, such as
Fletcher-Reeves(FR) \cite{fletcher1964function}, Polak-Ribi{\`e}re-Polak(PRP) \cite{polyak1969conjugate}, Hestenes-Stiefel(HS) \cite{hestenes1952methods}, Dai-Yuan(DY) \cite{dai1999nonlinear} and so on \cite{hu1991efficient,perry1978modified}:
\begin{eqnarray}
	&& \beta_{i}^{t,FR}=\frac{\Vert \m{g}^t_i\Vert^2}{\Vert \m{g}^{t-1}_i\Vert^2}  \qquad \quad \mbox{ and } \quad 
	 \beta_{i}^{t,PRP}=\frac{(\m{g}^t_i)\tr \m{y}^{t-1}_i}{\Vert \m{g}^{t-1}_i \Vert ^2}, \label{PRP}\\
&&	\beta_{i}^{t,HS}=\frac{(\m{g}^t_i)\tr \m{y}^{t-1}_i}{(\m{d}^{t-1}_{i})\tr\m{y}^{t-1}_i}  \quad \mbox{ and } \quad 
	\beta_{i}^{t,DY}=\frac{\Vert \m{g}^t_i \Vert^2}{(\m{d}^{t-1}_{i})\tr\m{y}^{t-1}_i},\label{DY}
\end{eqnarray}
where $\m{y}^t_{i}=\m{g}^{t+1}_i-\m{g}^t_i$. 

\begin{remark}
	Note that SDCG from \eqref{dcg1} and \eqref{dcg2} will be reduced to a decentralized gradient descent (DGD) method when $\beta_i^t=0$ for all $i$ and $t$,
       and will become a centralized nonlinear CG method if $\tilde{\m{W}}$ is the identity matrix and $n=1$.
\end{remark}

Following our previously defined notations, we can rewrite \eqref{dcg1} and \eqref{dcg2} for all the nodes together as
\begin{align} 
	& \m{x}^{t+1}=\m{W} \m{x}^{t}+\alpha \m{d}^{t}, \label{DCG1}\\ 
	& \m{d}^{t+1}=-\m{g}^{t+1}+\bm{\beta}^{t+1} \m{d}^{t},\label{DCG2}
\end{align}
where
\begin{equation*}
	{\bm{\beta}^{t+1}} = \left[ {\begin{array}{*{20}{c}}
			{{ \beta_1^{t+1}}\m{I}_p
			}&{}&{}\\
			{}&\ddots&{}\\
			{}&{}&{\beta_n^{t+1}}\m{I}_p
	\end{array}} \right] \in \R^{np \times np}.
\end{equation*}

\begin{figure*}[!t]
	\centering
	\subfloat[]{\includegraphics[width=2.5in]{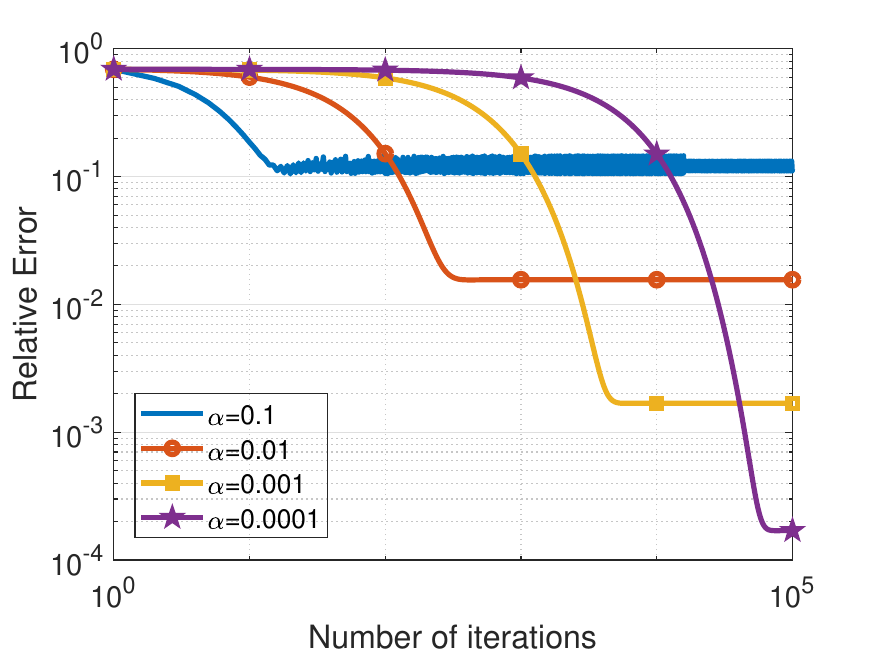}%
		\label{DCG_rel}}
	\hfil
	\subfloat[]{\includegraphics[width=2.5in]{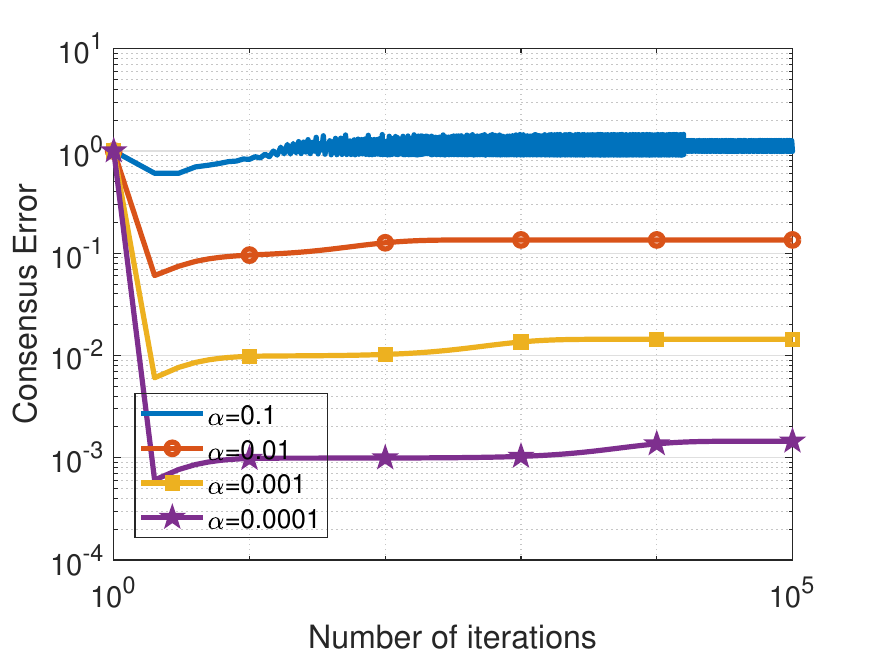}%
		\label{DCG_cons}}
	\caption{Relative error and consensus error of SDCG using $\beta_{i}^{t,PRP}$ versus iterations for stepsizes 0.1, 0.01, 0.001, and 0.0001. Relative error is given by \eqref{rel_error} and consensus error is defined as $\|\m{x}^t-\m{M}\m{x}^t\|$.}
	\label{dcg_stepsize}
\end{figure*}
Let us now apply SDCG to minimize the linear regression problem \eqref{linear_problem} with numerical results given in Fig.~\ref{dcg_stepsize},
where we can see that SDCG has a similar performance to DGD when $\beta_{i}^t=\beta_{i}^{t,PRP}$ in \eqref{PRP} and the convergence accuracy depends on the stepsize. 
However, SDCG  may not converge when the FR, DY and HS conjugate parameters are used,
even in the centralized setting with $n=1$ and $\tilde{\m{W}}=[1]$ \cite{dai2000nonlinear}.  
One reason on the practical convergence of SDCG with $\beta_{i}^t=\beta_{i}^{t,PRP}$ 
could be due to its restarting properties explained as below.
When the iterate jams, the difference $\|\m{x}^{t}_i-\m{x}^{t-1}_i\|$ tends to zero.
Then,  the gradient variation $\|\m{g}^{t}_i-\m{g}^{t-1}_i\|$ goes to zero 
by the Lipschitz condition \eqref{3.1}. However, the local gradient $\m{g}^t_i$ may not go to zero 
even the iterates converge to a stationary point where the total gradient is zero.
Thus, it implies by \eqref{PRP} that the PRP conjugate parameter $\beta_i^{t,PRP}$ will go to zero
 and the direction $\m{d}_i^t$ will tend to be the negative gradient direction $-\m{g}_i^t$, 
 which can be seen from Fig.~\ref{beta_iter}.
On the other hand, for $\beta_{i}^t=\beta_{i}^{t,DY}$ given by \eqref{DY},
the numerator $\Vert \m{g}^t_i \Vert^2$ may not go zero, while the denominator 
$(\m{d}^{t-1}_{i})\tr (\m{g}^{t}_i-\m{g}^{t-1}_i)$ will go to zero if $\m{d}_i^t$ is bounded.
Hence, the DY conjugate parameter $\beta_{i}^{t,DY}$ goes infinity, which
may explain why SDCG using DY formula \eqref{DY} usually diverges.
Based on the above observation, we think there are critical deficiencies 
in the convergence proofs of the decentralized CG methods proposed in \cite{xu2020distributed,chen2023decentralized}.
Another decentralized CG method with constant stepsize was proposed in \cite{shorinwa2023distributed}.
However, its convergence proof requires the local conjugate direction converges
to its mean over all nodes, which generally does not hold.

\begin{figure*}[!t]
	\centering
	\includegraphics[width=2.5in]{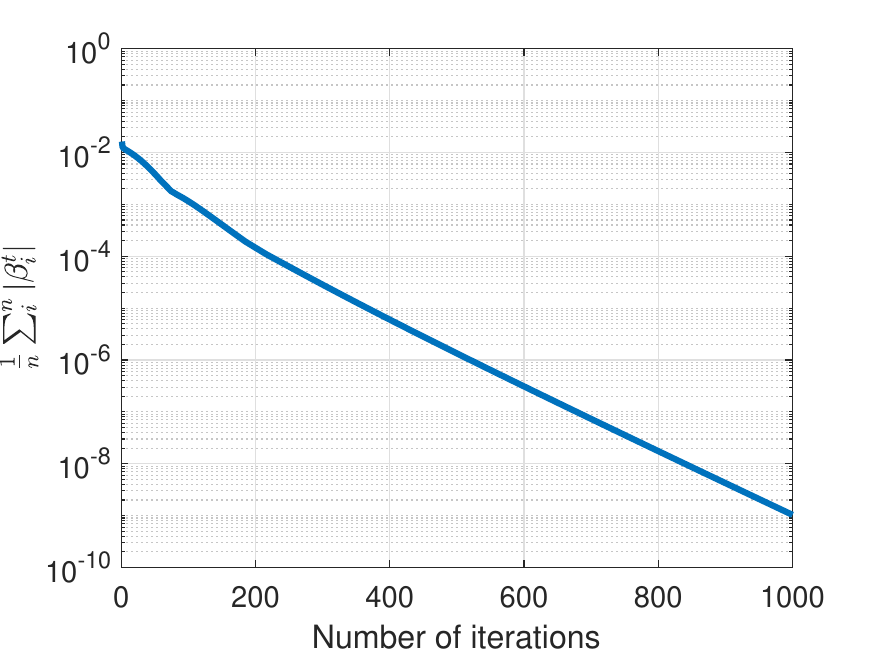}%
	\caption{$\frac{1}{n} \sum_i^n |\beta_{i}^t|$ generated by SDCG versus iterations for $\beta_{i}^t=\beta_{i}^{t,PRP}$.}
	\label{beta_iter}
\end{figure*}

Again, as seen in Fig.~\ref{dcg_stepsize}, although SDCG with using the PRP formula \eqref{PRP} could converge with constant stepsize, it does not converge to the exact solution of the original problem.
On the other hand, Dai \cite{dai2011convergence} shows that the traditional PRP nonlinear CG
with constant stepsize can have global convergence in the centralized setting.
Hence, we ask the following question:
{\it Can we design a decentralized CG algorithm with constant stepsize
such that it guarantees global convergence under similar mild conditions used by the traditional 
centralized nonlinear CG algorithm?}

\subsection{New Decentralized CG Method}
In this subsection, we provide a positive answer to the question raised in the previous section. 
Note that GT-based algorithms achieve the exact convergence 
by dynamically averaging the local gradients.
We would use similar techniques to eliminate the effect of local gradient dissimilarity.
Specifically, we introduce a variable $\m{v}_i$ on node $i$ to track the local gradient average by a dynamical average consensus technique \cite{sayed2014diffusion}.
Initializing  $\m{v}_i^0=\m{g}_i^0$, we use the ATC structure to update  $\m{v}_i^{t+1}$ as
\begin{equation}\label{tracking}
		\m{v}_i^{t+1}=\sum_{j \in \mathcal{N}_i}\tilde{W}_{ij}(\m{v}_j^{t}+\m{g}_j^{t+1}-\m{g}_j^{t}).
\end{equation}
By the above formula \eqref{tracking} and Lemma~\ref{property W}, it follows from induction that
\begin{equation}\label{vequalg}
\m{M}{\m{v}}^t=\m{M}{\m{g}}^t \quad \Longleftrightarrow \quad \bar{\m{v}}^t=\bar{\m{g}}^t,
\end{equation} 
which shows the average of $\m{v}_i^t$ over $i$ is the same as the local gradient average.
With this udpation of $\m{v}_i^t$, our new decentralized CG (NDCG) algorithm is given in Alg.~\ref{alg:Framwork+}.
In NDCG, we propose a new PRP-like CG parameter with respect to node $i$ as
\begin{equation}\label{new_beta}
	\beta_i^{t+1}=\frac{(\tilde{\m{v}}^{t+1}_i)\tr ({\m{g}}^{t+1}_i-{\m{g}}^t_i)}{\Vert \tilde{\m{v}}^t_i \Vert ^2},
\end{equation}
where $\tilde{\m{v}}_i^{t} = \m{v}_i^{t}+\frac{1}{\alpha}(\m{x}_i^{t}-\sum_{j \in \mathcal{N}_i} \tilde{W}_{i j} \m{x}_j^{t})$.
Combining $\beta_i^{t+1}$ with $\tilde{\m{v}}_i^{t+1}$, in NDCG we propose the conjugate search direction as 
\begin{equation}\label{new_dre}
	\tilde{\m{d}}_i^{t+1}=-\tilde{\m{v}}_i^{t+1}+\beta_i^{t+1} \tilde{\m{d}}_i^{t}.
\end{equation}
which has the PRP-type property that  when small steps are generated, $\beta_i^t$ will become small and 
$\tilde{\m{d}}_i^t$ will tend to  $-\tilde{\m{v}}_i^t$.

\begin{algorithm}[htb]
	\caption{NDCG with respect to node $i$}
	\label{alg:Framwork+}
	\begin{algorithmic}[1]
		\Require
		Initial point $\m{x}_i^0$,  Maximum iteration T, Stepsize $\alpha>0$,
		 Mixing matrix $\tilde{\m{W}}$.
		\State Set $t=0$,  $\m{v}_i^0=\m{g}_i^0$ and $\tilde{\m{d}}_i^{0}=-\tilde{\m{v}}_i^{0}=-{\m{g}}_i^{0}-\frac{1}{\alpha}(\m{x}_i^{0}-\sum_{j \in \mathcal{N}_i} \tilde{W}_{i j} \m{x}_j^{0})$.
		\State If $t \geq T$, stop.
		\State $\m{x}_i^{t+1}=\m{x}_i^{t}+\alpha \tilde{\m{d}}_i^{t}$.
		\State 	$	\m{v}_i^{t+1}=\sum_{j \in \mathcal{N}_i}\tilde{W}_{ij}(\m{v}_j^{t}+\m{g}_j^{t+1}-\m{g}_j^{t})$.
		\State $\tilde{\m{v}}_i^{t+1} = \m{v}_i^{t+1}+\frac{1}{\alpha}(\m{x}_i^{t+1}-\sum_{j \in \mathcal{N}_i} \tilde{W}_{i j} \m{x}_j^{t+1})$.
		\State $\beta_i^{t+1}=\frac{(\tilde{\m{v}}^{t+1}_i)\tr ({\m{g}}^{t+1}_i-{\m{g}}^t_i)}{\Vert \tilde{\m{v}}^t_i \Vert ^2}$.
		\State $\tilde{\m{d}}_i^{t+1}=-\tilde{\m{v}}_i^{t+1}+\beta_i^{t+1} \tilde{\m{d}}_i^{t}$. 
		\State Set $t=t+1$ and go to Step 2.
		\Ensure
		$\m{x}^T$.
	\end{algorithmic}
\end{algorithm}
\begin{remark} For the NDCG Alg.~\ref{alg:Framwork+}, we have the following comments: 
	\begin{itemize}
		\item[(a)] 
                 We can rewrite the update of $\m{x}^t$ in NDCG as
		\begin{eqnarray}\label{ndcg-moment}
			\m{x}^{t+1}=\m{W}\m{x}^{t}-\alpha\m{v}^t+\bm{\beta}^t(\m{x}^t-\m{x}^{t-1}),			
		\end{eqnarray}
                where $\bm{\beta}^t$ is a proper diagnal matrix given by $\beta_i^t$.
                Then, NDCG can be viewed as a decentralized heavy ball method with an adaptive momentum parameter and 
                is equivalent to the ABm method \cite{xin2019distributed} when $\bm{\beta}^t$ is fixed. 
               If $\beta_i^t=0$ for all $i$ and $t$, the update of $\m{x}^t$ is reduced as
			$\m{x}^{t+1}=\m{W}\m{x}^{t}-\alpha\m{v}^t$,
		which is equivalent to some semi-ATC GT method \cite{hong2022divergence,alghunaim2022unified}.
		 \item[(b)] 
		In the case $n=1$, we have $\tilde{\m{W}}=[1]$. The updates of NDCG would become
		\begin{align*}
			\m{x}_1^{t+1}&=\m{x}_1^{t}+\alpha \tilde{\m{d}}_1^{t} \quad \mbox{ and }  \quad  \tilde{\m{d}}_1^{t+1} =-{\m{g}}_1^{t+1}+\beta_1^{t+1} \tilde{\m{d}}_1^{t},\\
			\beta_1^{t+1}&=\frac{({\m{g}}^{t+1}_1)\tr ({\m{g}}^{t+1}_1-{\m{g}}^t_1)}{\Vert {\m{g}}^t_1 \Vert ^2},
		\end{align*}
		which is exactly the PRP nonlinear CG method proposed by Dai \cite{dai2011convergence}.
	\end{itemize}
\end{remark}

\subsubsection{Convergence of NDCG}
We now analyze the global convergence of NDCG for minimizing \eqref{obj_fun1}, where the local objective function $f_i$, $i=1,\ldots,n$,
is Lipschitz continuously differentiable, but possibly nonconvex. 
Recall that
$ \overline{\nabla} f(\m{x}^t)=\frac{1}{n} \sum_{i=1}^n \nabla f_i(\m{x}_i^t)$.
Then, it holds that $\m{1}\otimes\overline{\nabla} f({\m{x}}^t)=\m{M}{\m{g}}^t$. 
Note that if $\m{x}^* \in \R^{np}$ satisfies 
\begin{equation}\label{stationarity}
	\|\overline{\nabla} f(\m{x}^*)\|^2+\|\m{x}^*-\m{M}\m{x}^*\|^2=0,
\end{equation}
then we have $\m{x}^*_1= \ldots = \m{x}^*_i =: \m{z}^* $ and $\m{z}^* $ would be a first-order stationary point of  \eqref{obj_fun1}.
Since $\tilde{\m{v}}^t=\m{v}^t+\frac{1}{\alpha}(\m{I}-\m{W})\m{x}^t$ and $\m{M}\m{v}^t=\m{M}\m{g}^t$ by  \eqref{vequalg}, which implies $\m{M}\tilde{\m{v}}^t=\m{M}\m{g}^t$ due to $\m{M}(\m{I}-\m{W})=\m{0}$, we have 
\begin{align*}
	\|\m{1}\otimes\overline{\nabla} f({\m{x}}^t)\|=\|\m{M}\m{g}^t\|=\|\m{M}\tilde{\m{v}}^t\|\leq \|\tilde{\m{v}}^t\|.
\end{align*}
Hence, by \eqref{stationarity}, we say $\m{x}^t$ is an $\epsilon$-stationary solution for some $\epsilon >0$ if
\begin{equation}\label{eps-stationary}
 \|\tilde{\m{v}}^t\|^2+\|\m{x}^t-\m{M}\m{x}^t\|^2 \le \epsilon.
\end{equation}
Given any $\epsilon >0$, to show NDCG will generate a ($\m{x}^t, \tilde{\m{v}}^t)$ satisfying \eqref{eps-stationary},
we define the following potential function 
\begin{equation}\label{Potential-P}
	P(\m{x}^{t},\m{v}^{t})=F(\bar{\m{x}}^t)+\frac{1}{2\alpha n} \|\m{x}^{t}\|^2_{\m{I}-\m{W}}+\|\m{x}^{t}-\m{M}\m{x}^{t}\|^2+\|\m{v}^{t}-\m{M}\m{v}^{t}\|^2.
\end{equation}
Note $F(\bar{\m{x}}^t)$ is the objective function value of  \eqref{obj_fun1} at $\bar{\m{x}}^t$,
$\|\m{x}^{t}-\m{M}\m{x}^{t}\|^2$ represents the consensus error at $\m{x}^t$,
$ \|\m{x}^{t}\|^2_{\m{I}-\m{W}}$  describes the distance between $\m{x}_i^t$ and its neighbor's average,
and $\|\m{v}^{t}-\m{M}\m{v}^{t}\|^2$ gives the distance between the average gradient approximation $\m{v}^{t}$ with
 the average gradient $\m{M}\m{g}^t=\m{M}\m{v}^t$, called gradient tracking error. 
Obviously, $P$ is bounded below since each $f_i$ is bounded below by Assumption~\ref{as0}. 
We would show a sufficient decrease property of $P(\m{x}^{t},\m{v}^{t})$. 
We first deduce the bounds for $ -(\tilde{\m{v}}_i^t) \tr \tilde{\m{d}}_i^t $ and 
$\|\tilde{\m{d}}_i^t\|$ by applying the following \cite[Lemma 2.2.]{dai2011convergence}.
\begin{lemma}\label{lem2.5a}
	 Assume $\alpha$ is some constant in $(0,\frac{1}{4L}]$. Define the sequence $\left\{\xi^t\right\}$ as
 \begin{equation}\label{def-xi}
	\xi^0=1 ; \quad \xi^{t+1}=1+L \alpha (\xi^t)^2, \quad t \geq 0 .
 \end{equation}
	Then, we have $ 1 \leq \xi^t<c $ for all $ t \geq 0$,
	where $ c \in (0,2] $ is a constant given as
\begin{equation}\label{const-c}
	c=2\left(1+\sqrt{1-4L \alpha}\right)^{-1}.
\end{equation}
\end{lemma}

\begin{lemma}\label{lem2.5}
	Suppose Assumption \ref{as0} holds. Consider the NDCG algorithm with  $ \alpha \in(0,\frac{1}{4L}]$. Then, for all $t \geq 0$ and any $i$, we have
	\begin{align}
		&\left(2-\xi^t\right)\left\|\tilde{\m{v}}_i^t\right\|^2  \leq -(\tilde{\m{v}}_i^t)\tr\tilde{\m{d}}_i^t \leq \xi^t\left\|\tilde{\m{v}}_i^t\right\|^2, \label{descent1}\\
		&\left(2-\xi^t\right)\left\|\tilde{\m{v}}_i^t\right\| \leq\left\|\tilde{\m{d}}_i^t\right\| \leq \xi^t\left\|\tilde{\m{v}}_i^t\right\|\label{descent2},
	\end{align}
	where $\xi^t$ is the sequence defined in \eqref{def-xi}.
\end{lemma}
\begin{proof}
	Since $\tilde{\m{d}}_i^0=-\tilde{\m{v}}_i^0$ and $\xi^0=1$, \eqref{descent1} and \eqref{descent2} clearly hold for $t=0$. 
Assume that \eqref{descent1} and \eqref{descent2} hold for some $t$. Then by the Cauchy-Schwartz inequality,  $\m{x}_i^{t+1}=\m{x}_i^{t}+\alpha \tilde{\m{d}}_i^{t}$ and \eqref{3.1}, 
we have
	\begin{align}\label{lem2.5_eq1}
		\|\tilde{\m{d}}_i^{t+1}+\tilde{\m{v}}_i^{t+1}\|=&\|\beta_i^{t+1} \tilde{\m{d}}_i^{t}\|
		\leq \frac{\|{\m{g}}^{t+1}_i-{\m{g}}^{t}_i\| \|\tilde{\m{v}}^{t+1}_i\|}{\|\tilde{\m{v}}^{t}_i\|^2}\|\tilde{\m{d}}_i^{t}\| \\
		\leq &\frac{L\alpha \|\tilde{\m{d}}_i^{t}\|^2}{\|\tilde{\m{v}}^{t}_i\|^2}\|\tilde{\m{v}}^{t+1}_i\|
		\leq L\alpha (\xi^t)^2\|\tilde{\m{v}}^{t+1}_i\|.\notag
	\end{align}
	By the triangular inequality and the above relation, we obtain
	\begin{align*}
		\|\tilde{\m{d}}_i^{t+1}\|\leq \|\tilde{\m{v}}_i^{t+1}\|+\|\tilde{\m{d}}_i^{t+1}+\tilde{\m{v}}_i^{t+1}\|\leq \left(1+L\alpha (\xi^t)^2\right)\|\tilde{\m{v}}^{t+1}_i\|
	\end{align*}
	and
	\begin{align*}
		\|\tilde{\m{d}}_i^{t+1}\|\geq \|\tilde{\m{v}}_i^{t+1}\|-\|\tilde{\m{d}}_i^{t+1}+\tilde{\m{v}}_i^{t+1}\|
		\geq \left(1-L\alpha (\xi^t)^2\right)\|\tilde{\m{v}}^{t+1}_i\|.
	\end{align*}
	Then, it follows from the recursion  \eqref{def-xi} of $\xi^t$ and the choice of $\alpha$ that \eqref{descent2} holds for $t+1$. 
       By \eqref{lem2.5_eq1} and the Cauchy-Schwartz inequality, we have that
	\begin{align*}
		\vert (\tilde{\m{v}}_i^{t+1})\tr\tilde{\m{d}}_i^{t+1}+\|\tilde{\m{v}}_i^{t+1}\|^2 \vert 
		\leq \|\tilde{\m{v}}_i^{t+1}\| \|\tilde{\m{d}}_i^{t+1}+\tilde{\m{v}}_i^{t+1}\|
		\leq  L\alpha (\xi^t)^2\|\tilde{\m{v}}^{t+1}_i\|^2.
	\end{align*}
	Similarly, by the recursion  \eqref{def-xi} of $\xi^t$ and the choice of $\alpha$, \eqref{descent1} holds for $t+1$.
\end{proof}

We now derive a recursion relationship of 
$F(\bar{\m{x}}^t)+\frac{1}{2\alpha n} \|\m{x}^{t}\|^2_{\m{I}-\m{W}}$.
\begin{lemma}
	Suppose Assumption~\ref{as0} holds. If $ \alpha \in(0,\frac{2\sqrt{5}-4}{5L}]$, then for all $t \geq 0$ we have
	\begin{align}
		F(\bar{\m{x}}^{t+1})+\frac{1}{2\alpha n} \|\m{x}^{t+1}\|^2_{\m{I}-\m{W}}\leq& F(\bar{\m{x}}^t)+\frac{1}{2\alpha n} \|\m{x}^{t}\|^2_{\m{I}-\m{W}}-\left(\frac{\alpha(5-2\sqrt{5})}{4 n}-\frac{5L\alpha^2}{8n}\right)\|\tilde{\m{v}}^t\|^2 \notag \\
&+\frac{16L^2\alpha}{n}\|\m{x}^t-\m{M}\m{x}^t\|^2+\frac{16\alpha}{n}\|{\m{v}}^t-\m{M}{\m{v}}^t\|^2. \label{term1}
	\end{align} 
\end{lemma}

\begin{proof}
	By Assumption \ref{as0},  $F$ is Lipschitz continuously differentiable with Lipschitz constant $L$, which gives
	\begin{align*}
		F(\bar{\m{x}}^{t+1})\leq F(\bar{\m{x}}^t)+\Big\langle\nabla F(\bar{\m{x}}^t),\bar{\m{x}}^{t+1}-\bar{\m{x}}^{t}\Big\rangle+\frac{L}{2}\|\bar{\m{x}}^{t+1}-\bar{\m{x}}^{t}\|^{2}.
	\end{align*}
	Since $\rho(\m{I}-\m{W})<2$, we  have
	\begin{align*}
		\frac{1}{2\alpha n} \|\m{x}^{t+1}\|^2_{\m{I}-\m{W}} \leq \frac{1}{2\alpha n} \|\m{x}^{t}\|^2_{\m{I}-\m{W}}+\frac{1}{n} \Big\langle \frac{1}{\alpha} (\m{I}-\m{W})\m{x}^t, \m{x}^{t+1}-\m{x}^t\Big\rangle + \frac{1}{2\alpha n}\|\m{x}^{t+1}-\m{x}^t\|^2.
	\end{align*}
In addition, we have
\begin{align*}
	&\frac{L}{2}\|\bar{\m{x}}^{t+1}-\bar{\m{x}}^{t}\|^{2}=\frac{L}{2n}\|\m{1}\otimes\bar{\m{x}}^{t+1}-\m{1}\otimes\bar{\m{x}}^{t}\|^{2}=\frac{L}{2n}\|\m{M}{\m{x}}^{t+1}-\m{M}{\m{x}}^{t}\|^{2}\leq\frac{L}{2n}\|{\m{x}}^{t+1}-{\m{x}}^{t}\|^{2}.
\end{align*}
	Adding the above three inequalities together and using Lemma~\ref{lem2.5} yields 
	\begin{eqnarray}\label{add}
	   & & F(\bar{\m{x}}^{t+1})+\frac{1}{2\alpha n} \|\m{x}^{t+1}\|^2_{\m{I}-\m{W}} \\
  &\leq& F(\bar{\m{x}}^t)+\frac{1}{2\alpha n} \|\m{x}^{t}\|^2_{\m{I}-\m{W}}+\left(\frac{L\alpha^2c^2}{2n}+\frac{\alpha c^2}{2 n}\right) \|\tilde{\m{v}}^t\|^2 + \mbox{RemainTerm} \notag 
	\end{eqnarray}
	where 
	\begin{align}
		&\mbox{RemainTerm} = \Big\langle\nabla F(\bar{\m{x}}^t),\bar{\m{x}}^{t+1}-\bar{\m{x}}^{t}\Big\rangle + \frac{1}{n} \Big\langle \frac{1}{\alpha} (\m{I}-\m{W})\m{x}^t, \m{x}^{t+1}-\m{x}^t\Big\rangle \notag\\
		=&\frac{1}{n} \sum_{i=1}^n \Big\langle\nabla F(\bar{\m{x}}^t),{\m{x}}_i^{t+1}-{\m{x}}_i^{t}\Big\rangle+\frac{1}{n} \sum_{i=1}^n \Big\langle \frac{1}{\alpha} (\m{x}_i^{t}-\sum_{j \in \mathcal{N}_i} \tilde{W}_{i j} \m{x}_j^{t}), \m{x}_i^{t+1}-\m{x}_i^t\Big\rangle \notag\\
		=&\underbrace{\frac{1}{n} \sum_{i=1}^n \Big\langle\tilde{\m{v}}_i^t,{\m{x}}_i^{t+1}-{\m{x}}_i^{t}\Big\rangle}_{\text {term (A)}}
		+\underbrace{\frac{1}{n} \sum_{i=1}^n \Big\langle\nabla F(\bar{\m{x}}^t)-\overline{\nabla} f({\m{x}}^t),{\m{x}}_i^{t+1}-{\m{x}}_i^{t}\Big\rangle}_{\text {term (B)}}\notag\\
		&+\underbrace{\frac{1}{n} \sum_{i=1}^n \Big\langle\overline{\nabla} f({\m{x}}^t)-{\m{v}}_i^t,{\m{x}}_i^{t+1}-{\m{x}}_i^{t}\Big\rangle}_{\text {term (C)}}. \label{terms}
	\end{align}
	For term (A), by Lemma~\ref{lem2.5} and $\m{x}_i^{t+1}=\m{x}_i^{t}+\alpha \tilde{\m{d}}_i^{t}$, we have
	\begin{align}\label{termA}
		\frac{1}{n} \sum_{i=1}^n \Big\langle\tilde{\m{v}}_i^t,{\m{x}}_i^{t+1}-{\m{x}}_i^{t}\Big\rangle \leq -\frac{\alpha(2-c)}{ n} \|\tilde{\m{v}}^t\|^2.
	\end{align}
	For term (B), by Young's inequality with constant $b >0$ and  $\m{1}\otimes\bar{\m{x}}^t=\m{M}{\m{x}}^t$, we have 
	\begin{eqnarray}\label{termB}
	  &&	\frac{1}{n} \sum_{i=1}^n \Big\langle\nabla F(\bar{\m{x}}^t)-\overline{\nabla} f({\m{x}}^t),{\m{x}}_i^{t+1}-{\m{x}}_i^{t}\Big\rangle \\
  &\leq& \frac{b}{2}\|\nabla F(\bar{\m{x}}^t)-\overline{\nabla} f({\m{x}}^t)\|^2+\frac{1}{2nb}\|{\m{x}}^{t+1}-{\m{x}}^{t}\|^2 \notag\\
		&\leq & \frac{L^2b}{2n}\|\m{1}\otimes\bar{\m{x}}^t-\m{x}^t \|^2+\frac{2\alpha^2}{nb}\|\tilde{\m{v}}^t\|^2
	   =\frac{L^2b}{2n}\|\m{M}\m{x}^t-\m{x}^t \|^2+\frac{2\alpha^2}{nb}\|\tilde{\m{v}}^t\|^2,\notag
	\end{eqnarray}
where the last inequality uses Lemma~\ref{important}, Lemma~\ref{lem2.5},  $L$-Lipschitz continuity of $\nabla F$.
	For term (C), by Young's inequality with constant $d>0$, we have
	\begin{align}\label{termC}
		\frac{1}{n} \sum_{i=1}^n \Big\langle\overline{\nabla} f({\m{x}}^t)-{\m{v}}_i^t,{\m{x}}_i^{t+1}-{\m{x}}_i^{t}\Big\rangle \leq &\frac{d}{2n}\|\m{1}\otimes\overline{\nabla} f({\m{x}}^t)-{\m{v}}^t\|^2+\frac{1}{2nd}\|{\m{x}}^{t+1}-{\m{x}}^{t}\|^2 \notag \\
		\leq & \frac{d}{2n}\|{\m{v}}^t-\m{M}{\m{v}}^t\|^2+\frac{2\alpha^2}{nd}\|\tilde{\m{v}}^t\|^2,
	\end{align}
where the last inequality is due to Lemma~\ref{lem2.5} and the relation $\m{1}\otimes\overline{\nabla} f({\m{x}}^t)=\m{M}{\m{v}}^t$ by \eqref{vequalg}.
	Substituting \eqref{termA}, \eqref{termB}, and \eqref{termC} into \eqref{terms}, we obtain
	\begin{align*}
		&\Big\langle\nabla F(\bar{\m{x}}^t),\bar{\m{x}}^{t+1}-\bar{\m{x}}^{t}\Big\rangle+\frac{1}{n} \Big\langle \frac{1}{\alpha} (\m{I}-\m{W})\m{x}^t, \m{x}^{t+1}-\m{x}^t\Big\rangle\\
		\leq &-\left(\frac{\alpha(2-c)}{ n}-\frac{2\alpha^2}{n}\left(\frac{1}{b}+\frac{1}{d}\right)\right)\|\tilde{\m{v}}^t\|^2
		+\frac{L^2b}{2n}\|\m{x}^t-\m{M}\m{x}^t\|^2+\frac{d}{2n}\|{\m{v}}^t-\m{M}{\m{v}}^t\|^2.
	\end{align*}
	By substituting the above inequality into \eqref{add}, it holds that
	\begin{eqnarray*}
	&& F(\bar{\m{x}}^{t+1})+\frac{1}{2\alpha n} \|\m{x}^{t+1}\|^2_{\m{I}-\m{W}} \\ 
      &\leq& F(\bar{\m{x}}^t)+\frac{1}{2\alpha n} \|\m{x}^{t}\|^2_{\m{I}-\m{W}}
	-\left(\frac{\alpha(4-2c-c^2)}{ 2n}-\frac{2\alpha^2}{n}\left(\frac{1}{b}+\frac{1}{d}\right)-\frac{L\alpha^2c^2}{2n}\right)\|\tilde{\m{v}}^t\|^2\\
	&&+\frac{L^2b}{2n}\|\m{x}^t-\m{M}\m{x}^t\|^2+\frac{d}{2n}\|{\m{v}}^t-\m{M}{\m{v}}^t\|^2.
\end{eqnarray*} 
Since $0<\alpha \leq \frac{2\sqrt{5}-4}{5L}$, we have from the definition of constant $c$ in \eqref{const-c} that 
\begin{equation}\label{new-bd-c}
1<c\leq \sqrt{5}/2, 
\end{equation}
 which implies $\frac{11}{4}-\sqrt{5}\leq 4-2c-c^2 \leq 1$. Then, setting $b=d=32\alpha$ yields \eqref{term1}.
\end{proof}

\begin{theorem}\label{nonconvex_convergence}
	Suppose Assumption \ref{as0} holds. If
	\begin{align}\label{alpha}
	0 <	\alpha < \min \left\{\frac{(1-{\sigma}^2)n}{32L^2},\frac{2(5-2\sqrt{5})(1-{\sigma}^2)}{4n(5\sigma^2L^2+18-8\sqrt{5})+5L},\frac{2\sqrt{5}-4}{5L}\right\},
	\end{align}
	then we have  
	\begin{equation}\label{cg-converge-rate}
	 \min_{0 \le t \le T} \left\{ \|\tilde{\m{v}}^t\|^2+\|{\m{v}}^t-\m{M}{\m{v}}^t\|^2+\|\m{x}^{t}-\m{M}\m{x}^{t}\|^2 \right\}
		\leq \frac{P(\m{x}^{0};\m{v}^{0})-P(\m{x}^{T};\m{v}^{T})}{\gamma \min\{\alpha,1\}T},
	\end{equation}
	where $P$ is the potential function defined in \eqref{Potential-P} and
	$\gamma >0$ is a constant depending only on $n$, $L$, and $\sigma$.
\end{theorem}
\begin{proof}
	First, by using Young's inequality with constant $\eta >0$, we 
	 establish a recursion upper bound on $\|{\m{v}}^t-\m{M}{\m{v}}^t\|^2$.
	\begin{align}\label{term2}
		\|{\m{v}}^{t+1}-\m{M}{\m{v}}^{t+1}\|^2
		=&\| \m{W}{\m{v}}^{t}+\m{W}\m{g}^{t+1}-\m{W}\m{g}^{t} -\m{M}\m{v}^t-\m{M}\m{g}^{t+1}+\m{M}\m{g}^t\|^2 \notag\\
		\leq & (1+\eta)\|\m{W}{\m{v}}^t-\m{M}{\m{v}}^t\|^2+(1+1/\eta)\|(\m{W}-\m{M})(\m{g}^{t+1}-\m{g}^t)\|^2\notag\\
		\leq & (1+\eta)\sigma^2\|{\m{v}}^t-\m{M}{\m{v}}^t\|^2+(1+1/\eta)\sigma^2L^2\|\m{x}^{t+1}-\m{x}^t\|^2,\notag\\
		\leq & (1+\eta)\sigma^2\|{\m{v}}^t-\m{M}{\m{v}}^t\|^2+(1+1/\eta)\sigma^2L^2 c^2 \alpha^2\|\tilde{\m{v}}^t\|^2,\notag\\
		\leq & (1+\eta)\sigma^2\|{\m{v}}^t-\m{M}{\m{v}}^t\|^2+(1+1/\eta)(5/4)\sigma^2L^2\alpha^2\|\tilde{\m{v}}^t\|^2,
	\end{align}
where the second inequality uses Lemma~\ref{property W} and $L$-Lipschitz continuity of $\nabla f$,
the third inequality uses Lemma~\ref{lem2.5},
the last inequality applies  $\alpha \leq \frac{2\sqrt{5}-4}{5L}$ and \eqref{new-bd-c}.

 We now establish a recursion upper bound on $\|{\m{x}}^t-\m{M}{\m{x}}^t\|^2$.
  By rewriting $\m{x}^{t+1}=\m{W}\m{x}^t-\alpha\m{v}^t+\alpha \bm{\beta}^{t}\tilde{\m{d}}^{t-1}$
  and applying Young's inequality with constant $\tau >0$, we have
	\begin{align}\label{term3}
		\|\m{x}^{t+1}-\m{M}\m{x}^{t+1}\|^2
		=&\|\m{W}\m{x}^t-\m{M}\m{x}^t-\alpha\m{v}^t+\alpha\m{M}\m{v}^t+\alpha \bm{\beta}^{t}\tilde{\m{d}}^{t-1}-\alpha\m{M} \bm{\beta}^{t}\tilde{\m{d}}^{t-1}\|^2 \notag \\
		\leq &(1+\tau) \|\m{W}\m{x}^t-\m{M}\m{x}^t\|^2\notag+2(1+1/\tau)\alpha^2\|{\m{v}}^t-\m{M}{\m{v}}^t\|^2\notag\\
		&+2(1+1/\tau)\alpha^2\|(\m{I}-\m{M})\bm{\beta}^{t}\tilde{\m{d}}^{t-1}\|^2\notag\\
		\leq&(1+\tau){\sigma}^2 \|\m{x}^{t}-\m{M}\m{x}^{t}\|^2+2(1+1/\tau)\alpha^2\|{\m{v}}^t-\m{M}{\m{v}}^t\|^2 \notag \\
		&+2(1+1/\tau)\alpha^2\|\bm{\beta}^{t}\tilde{\m{d}}^{t-1}\|^2,\notag\\
		\leq&(1+\tau){\sigma}^2 \|\m{x}^{t}-\m{M}\m{x}^{t}\|^2+2(1+1/\tau)\alpha^2\|{\m{v}}^t-\m{M}{\m{v}}^t\|^2 \notag \\
		&+2(1+1/\tau)(9/4-\sqrt{5})\alpha^2\|\tilde{\m{v}}^{t}\|^2, 
	\end{align}	
where the second inequality uses Lemma \ref{property W} and  $\rho(\m{I}-\m{M}) \leq 1$,
and the last inequality uses
$\alpha \leq \frac{2\sqrt{5}-4}{5L}$, Lemma~\ref{lem2.5}, \eqref{lem2.5_eq1} and \eqref{new-bd-c}.
Adding up \eqref{term1}, \eqref{term2}, and \eqref{term3} yields
	\begin{align*}
		&P(\m{x}^{t+1},\m{v}^{t+1})\\
		\leq &	P(\m{x}^{t},\m{v}^{t})-\left(1-(1+\tau){\sigma}^2-\frac{16L^2\alpha}{n}\right)\|\m{x}^{t}-\m{M}\m{x}^{t}\|^2
		\\
		&-\left(1-(1+\eta)\sigma^2-2\left(1+\frac{1}{\tau}\right)\alpha^2-\frac{16\alpha}{n}\right)\|{\m{v}}^t-\m{M}{\m{v}}^t\|^2\notag\\
		&-\left(\frac{\alpha(5-2\sqrt{5})}{4 n}-\frac{5L\alpha^2}{8n}-\left(1+\frac{1}{\eta}\right)\frac{5}{4}\sigma^2L^2\alpha^2-2\left(1+\frac{1}{\tau}\right)\left(\frac{9}{4}-\sqrt{5}\right)\alpha^2\right)\|\tilde{\m{v}}^{t}\|^2.
	\end{align*}
Then, setting $\eta=\tau=\frac{1-{\sigma}^2}{2{\sigma}^2}$, we have
	\begin{align}\label{P_descent}
		&P(\m{x}^{t+1},\m{v}^{t+1})\\
		\leq 	&P(\m{x}^{t},\m{v}^{t})-\left(\frac{1-{\sigma}^2}{2}-\frac{16L^2\alpha}{n}\right)\|\m{x}^{t}-\m{M}\m{x}^{t}\|^2\notag\\
		&-\left(\frac{\alpha(5-2\sqrt{5})}{4 n}-\frac{5L\alpha^2}{8n}-\frac{5(1+\sigma^2)\sigma^2L^2\alpha^2}{4(1-\sigma^2)}-\frac{2(1+{\sigma}^2)\alpha^2}{1-{\sigma}^2}\left(\frac{9}{4}-\sqrt{5}\right)\right)\|\tilde{\m{v}}^{t}\|^2\notag\\
		&-\left(\frac{1-\sigma^2}{2}-\frac{2(1+{\sigma}^2)\alpha^2}{1-{\sigma}^2}-\frac{16\alpha}{n}\right)\|{\m{v}}^t-\m{M}{\m{v}}^t\|^2
		\notag\\
		\leq& P(,\m{x}^{t},\m{v}^{t})
		-\left(\frac{\alpha(5-2\sqrt{5})}{4 n}-\frac{4n(5\sigma^2L^2+18-8\sqrt{5})+5L}{8n(1-{\sigma}^2)}\alpha^2\right)\|\tilde{\m{v}}^{t}\|^2\notag\\
		&-\left(\frac{1-\sigma^2}{2}-\frac{4\alpha^2}{1-{\sigma}^2}-\frac{16\alpha}{n}\right)\|{\m{v}}^t-\m{M}{\m{v}}^t\|^2
		-\left(\frac{1-{\sigma}^2}{2}-\frac{16L^2\alpha}{n}\right)\|\m{x}^{t}-\m{M}\m{x}^{t}\|^2,\notag
	\end{align}
where the last inequality sues $0<\sigma<1$.
	Then, by the choice of $\alpha$ in \eqref{alpha}, we can derive from \eqref{P_descent}
	and direct calculations	that there exist a constant $\gamma >0$, depending only on 
	$n$, $L$, $\sigma$, such that
	\begin{equation*}
		P(\m{x}^{t+1},\m{v}^{t+1}) \leq P(\m{x}^{t},\m{v}^{t}) - \gamma \left( 
		\alpha \|\tilde{\m{v}}^t\|^2 +\|{\m{v}}^t-\m{M}{\m{v}}^t\|^2+\|\m{x}^{t}-\m{M}\m{x}^{t}\|^2 \right). 
	\end{equation*}
	Summing the above relation over $t=1, \ldots, T$ and dividing $\gamma \min\{\alpha,1\}$ from 
	both sides gives \eqref{cg-converge-rate}.
\end{proof}

\begin{remark} We have the following comments for NDCG.
	\begin{itemize}
		\item[(a)] Given any $\epsilon >0$, by Theorem~\ref{nonconvex_convergence},
		 NDCG will take at most $P(\m{x}^{0},\m{v}^{0})/(\gamma \min\{\alpha,1\} \epsilon)$ iterations 
		 to generate an $\epsilon$-stationary point $\m{x}^t$ satisfying \eqref{eps-stationary},
		 where $\gamma >0$ is the constant given in Theorem~\ref{nonconvex_convergence}.
		\item[(b)] NDCG converges under very mild conditions, which are same as those required 
		by the centralized PRP developed in \cite{dai2011convergence}.
		In addition, formula \eqref{ndcg-moment} shows NDCG can be viewed as a decentralized 
		heavy ball method using adaptive momentum parameter $\bm{\beta}^t$.  
        However, the convergence of the ABm method \cite{xin2019distributed}, which is 
        a heavy ball method using a fixed momentum term, is still unknown for nonconvex optimization.
		\item[(c)]  The convergence stepsize range of semi-ATC GT method in the nonconex deterministic setting \cite{hong2022divergence} is 
		\[
		 	\left(0,\mathcal{O}\left(\min \left\{\frac{(1-{\sigma}^2)^3n}{L^2},\frac{1-{\sigma}^2}{n+L}\right\}\right)\right),
		 \]     
		 while for networks with $\sigma < \sqrt{\frac{8(9-4\sqrt{5})n+5L}{20nL^2}}$,  the convergence     
         stepsize range of NDCG is 
 		 \begin{equation}\label{range}
		 	 \left(0,\mathcal{O}\left(\min \left\{\frac{(1-{\sigma}^2)n}{L^2},\frac{1-{\sigma}^2}{n+L}\right\}\right)\right),
		 \end{equation}   
		  and for networks with $1 > \sigma > \sqrt{\frac{8(9-4\sqrt{5})n+5L}{20nL^2}}$,
		  the convergence  stepsize range of NDCG is 
		 \[
		\left(0, \mathcal{O}\left(\min \left\{\frac{(1-{\sigma}^2)n}{L^2},\frac{1-{\sigma}^2}{n\sigma^2L^2}\right\}\right)\right).
		\]
	\end{itemize}
\end{remark}

\subsection{Decentralized Memoryless BFGS Method}
One key issue for decentralized quasi-Newton methods is to construct effective quasi-Newton
matrices using information on individual nodes to accelerate the convergence.
Different regularizing or damping techniques are used in \cite{eisen2017decentralized,eisen2019primal,li2021bfgs,zhang2023variance} 
to ensure the quasi-Newton matrices used on local nodes are positive definite and the overall
global convergence. In particular, for global convergence, 
the decentralized ADMM developed in \cite{li2021bfgs} requires explicit positive lower bound on 
the perturbation parameter of the quasi-Newton matrix.
In this section, we develop a new quasi-Newton method based on 
the self-scaling memoryless BFGS technique proposed by Shanno \cite{shanno1978convergence},
which could be also considered as a self-scaling three-term decentralized CG method.
Our proposed new method focuses on the strongly convex objective function. 
Hence, in this section, we also need the following strongly convex assumptions on the objective
function in \eqref{obj_fun1}.
\begin{assum}\label{as2}
	The local objective functions $\left\{ f_i(\m{z}) \right\}_{i = 1}^n $ are strongly convex with
	modulus $\mu>0$, that is, for any $ \m{z},\tilde{\m{z}} \in \R^p$ and $i = 1,\ldots,n$, we have
	\begin{equation}\label{3.2}
		{f_i}(\tilde{\m{z}}) \ge {f_i}(\m{z}) + \nabla {f_i}{(\m{z})\tr}(\tilde{\m{z}}- \m{z}) + \frac{\mu }{2}{\left\| {\tilde{\m{z}} - \m{z}} \right\|^2}.
	\end{equation}
\end{assum}
By Assumption~\ref{as0} and Assumption~\ref{as2}, 
 for any $ \m{z} \in \R^p$ and $i = 1,\ldots,n$, we have
\begin{equation}\label{3.3}
	\mu {\m{I}_p} \preceq {\nabla ^2}{f_i}(\m{z}) \preceq L {\m{I}_p}.
\end{equation}
Since the Hessian $\nabla^2 f(\m{x})$ is a block diagonal matrix whose $i$-th diagonal block is $\nabla^2 f_i(\m{x}_i)$, \eqref{3.3} also implies
that for all $ \m{x} \in {\R^{np}}$, we have
\begin{equation}\label{3.4}
	\mu {\m{I}} \preceq {\nabla ^2}f(\m{x}) \preceq L{\m{I}}.
\end{equation}
Since the original problem \eqref{obj_fun1} is strongly convex by Assumption~\ref{as2}, 
it has a unique global optimal solution $\m{z}^*$.
We now let $\Omega_i^t = \left\{\m{y} \in \R^{p}: (\m{s}_i^{t})\tr\m{y} >0 \right\}$
and define the function $H_i^t: \Omega_i^t  \rightarrow \R^{p \times p}$ as
\begin{eqnarray}\label{H0}
	H_i^t(\m{y}) &= &\tau_i^t\left(\m{I}_p-\frac{\m{s}_i^t(\m{y})\tr+\m{y}(\m{s}_i^t)\tr}{(\m{s}_i^t)\tr\m{y}}\right)
	+\left(1+\frac{\tau_i^t\|\m{y}\|^2}{(\m{s}_i^{t})\tr\m{y}}\right)\frac{\m{s}_i^t(\m{s}_i^t)\tr}{(\m{s}_i^{t})\tr\m{y}} \notag \\
	&=& \frac{(\m{s}_i^t)\tr \m{y}}{\|\m{y}\|^2} \m{I}_p - 
	\frac{\m{s}_i^t(\m{y})\tr+\m{y}(\m{s}_i^t)\tr}{\|\m{y}\|^2} + 
	2 \frac{\m{s}_i^t(\m{s}_i^t)\tr}{(\m{s}_i^{t})\tr\m{y}},
\end{eqnarray}
 where $\m{s}_i^{t}=\m{x}_i^{t+1}-\m{x}_i^t$ and 
$\tau_i^t =\frac{(\m{s}_i^t)\tr \m{y}}{\|\m{y}\|^2}$.
Then, our decentralized memoryless BFGS (DMBFGS) method is described
in Alg.~\ref{alg:Framwork1}.
\begin{algorithm}[htb]
	\caption{DMBFGS with respect to node $i$}
	\label{alg:Framwork1}
	\begin{algorithmic}[1]
		\Require
		Initial point $\m{x}_i^0$,  Maximum iteration T, Stepsize $\alpha>0$, 
		Mixing matrix $\tilde{\m{W}}$, Parameters $0<l \ll u$.
		\State Set $t=0$ and $\m{d}_i^{0}=-\m{v}_i^{0}=-\m{g}_i^{0}$.
		\State If $t \geq T$, stop.
		\State $\m{x}_i^{t+1}=\sum_{j \in \mathcal{N}_i}\tilde{W}_{ij} (\m{x}_j^{t}+\alpha \m{d}_j^{t})$.
		\State
		$	\m{v}_i^{t+1}=\sum_{j \in \mathcal{N}_i}\tilde{W}_{ij}(\m{v}_j^{t}+\m{g}_j^{t+1}-\m{g}_j^{t})$.
		\State 
		\begin{equation*}
			\m{y}_i^t=\left\{\begin{array}{cl}
				\check{\m{y}}_i^t, & \text { if }  (\m{s}_i^{t})\tr \check{\m{y}}_i^t >0 \mbox{ and }  [\lambda_{\min } (H_i^t(\check{\m{y}}_i^t)),\lambda_{\max } (H_i^t(\check{\m{y}}_i^t))] \subset  [l,u]; \\
				\hat{\m{y}}_i^t, & \text { otherwise. } 
			\end{array}\right.
		\end{equation*}
		where $\m{s}_i^{t}=\m{x}_i^{t+1}-\m{x}_i^t$, $\hat{\m{y}}_i^{t}=\m{g}_i^{t+1}-\m{g}_i^t$ and $\check{\m{y}}_i^{t}=\m{v}_i^{t+1}-\m{v}_i^t$.
		\State $\tau_i^t = \frac{(\m{s}_i^t)\tr \m{y}_i^t}{\|\m{y}_i^t\|^2}$ \mbox{ and } 
		 $\theta_i^{t+1}=\frac{(\m{v}_i^{t+1})\tr\m{s}_i^t}{\|\m{y}_i^t \|^2}$.
		\State $	\beta_i^{t+1}=\frac{(\m{v}_i^{t+1})\tr\m{y}_i^t}{\|\m{y}_i^t\|^2}-2\frac{(\m{v}_i^{t+1})\tr\m{s}_i^t}{(\m{s}_i^{t})\tr\m{y}_i^t}$.
		\State
		${\m{d}}_i^{t+1}=-\tau_i^t\m{v}_i^{t+1}+\beta_i^{t+1} {\m{s}}_i^{t}+\theta_i^{t+1}\m{y}_i^t$. 
		\State Set $t=t+1$ and go to Step 2.
		\Ensure
		$\m{x}^T$.
	\end{algorithmic}
\end{algorithm}

In Alg.~\ref{alg:Framwork1}, the search direction $\m{d}_i^{t+1}$ can be
considered as a self-scaling three-term decentralized CG direction, which can be also
 equivalently written as the following quasi-Newton direction
\begin{equation}\label{d_quasi}
	\m{d}_i^{t+1}=-\m{H}_i^{t+1}\m{v}_i^{t+1},
\end{equation}
where
\begin{align}\label{H1}
	\m{H}_i^{t+1} = H_i^t(\m{y}_i^t) = 
	\frac{(\m{s}_i^t)\tr \m{y}_i^t}{\|\m{y}_i^t\|^2} \m{I}_p - 
	\frac{\m{s}_i^t(\m{y}_i^t)\tr+\m{y}_i^t(\m{s}_i^t)\tr}{\|\m{y}_i^t\|^2} + 
	2 \frac{\m{s}_i^t(\m{s}_i^t)\tr}{(\m{s}_i^{t})\tr\m{y}_i^t}.
\end{align}
Step 5 in Alg.~\ref{alg:Framwork1} is to ensure the quasi-Newton matrix \eqref{H1} 
has positive bounded eigenvalues. 
Since $\m{v}_i^t$ is generated by the gradient tracking technique in Step 4,
it captures some information of the average gradients on different nodes.
So, we prefer using $\m{v}_i^t$ to generate the quasi-Newton matrix by $H_i^t(\check{\m{y}}_i^t)$
whenever possible, where $\check{\m{y}}_i^{t}=\m{v}_i^{t+1}-\m{v}_i^t$.
However, $H_i^t(\check{\m{y}}_i^t)$ is not necessarily positive definite and bounded. 
So, in Step 5 we simply check if the smallest and largest eigenvalues of the quasi-Newton
matrix $H_i^t(\check{\m{y}}_i^t)$ belong to an interval $[l,u]$, where 
$0 < l \ll u$ are two parameters. If not, we would use the alternative $\m{g}_i^t$ to 
update the quasi-Newton matrix by $H_i^t(\hat{\m{y}}_i^t)$, where
 $\hat{\m{y}}_i^{t}=\m{g}_i^{t+1}-\m{g}_i^t$.
By the Assumption~\ref{as0}, Assumption~\ref{as2} and the BFGS updating formula \eqref{H1},  $H_i^t(\hat{\m{y}}_i^t)$ is guaranteed to be positive definite and uniformly bounded.
We now explain the computational cost of the smallest and largest eigenvalues of
 $\m{H}_i^{t+1}$ is in fact negligible. 
Note that $\m{H}_i^{t+1} = H_i^t(\check{\m{y}}_i^t)$ is obtained 
 from the scalar matrix $\tau_i^t \m{I}_p$ by a rank-two BFGS update.
 Hence, if $\tau_i^t = \frac{(\m{s}_i^t)\tr \m{y}_i^t}{\|\m{y}_i^t\|^2}  >0 $,
  $\m{H}_i^{t+1}$ has $p-2$ eigenvalues of $\tau_i^t$ and 
 two eigenvalues $0< \lambda_i^{t+1} \le \Lambda_i^{t+1}$ satisfying
\begin{equation}\label{qua}
	\left\{\begin{array}{cl}
		\lambda_i^{t+1} \Lambda_i^{t+1} &= \frac{\|\m{s}_i^t\|^2}{\|\m{y}_i^t\|^2}, \\
		\lambda_i^{t+1}+ \Lambda_i^{t+1} &=\frac{2\|\m{s}_i^t\|^2}{(\m{s}_i^t)\tr \m{y}_i^t}.
	\end{array}\right.
\end{equation}
The system \eqref{qua} defines a quadratic equation, 
which has two roots
\begin{align}\label{lam_H}
	\left\{\begin{array}{cl}
		\lambda_i^{t+1} = \lambda_{\min } (H_i^t(\m{y}_i^t))
		=\frac{\|\m{s}_i^t\|^2}{(\m{s}_i^t)\tr \check{\m{y}}_i^t}\left(1- \sqrt{1- \frac{((\m{s}_i^t)\tr \check{\m{y}}_i^t)^2}{\|\m{s}_i^t\|^2\|\check{\m{y}}_i^t\|^2}} \right),\\
		~\\
\Lambda_i^{t+1}	=	\lambda_{\max } (H_i^t(\m{y}_i^t))
		= \frac{\|\m{s}_i^t\|^2}{(\m{s}_i^t)\tr \check{\m{y}}_i^t}\left(1+ \sqrt{1- \frac{((\m{s}_i^t)\tr \check{\m{y}}_i^t)^2}{\|\m{s}_i^t\|^2\|\check{\m{y}}_i^t\|^2}} \right).
	\end{array}\right.
\end{align}
We want to emphasize that although the Barzilai-Borwein (BB) scaling factor
$\tau_i^t = \frac{(\m{s}_i^t)\tr \m{y}_i^t}{\|\m{y}_i^t\|^2}$ is used 
in our  Alg.~\ref{alg:Framwork1}, any proper positive uniformly bounded 
scaling factor $\tau_i^t$ can be used in formula \eqref{H0} for constructing
$H_i^t(\m{y})$.

We have the following comments on the computational and memory cost comparison of DMBFGS 
with other decentralized quasi-Newton methods.
\begin{remark}
DMBFG only involves vector-vector products and uses minimum
 computational and memory cost per iteration, which are on the order of $\mathcal{O}(p)$.
The per-iteration computational and memory cost of 
D-LM-BFGS  \cite{zhang2023variance} is on the order of $\mathcal{O}(Mp)$ due to its two-loop recursion
computation of the search direction, 
while those of DR-LM-DFP \cite{zhang2023variance}  are $\mathcal{O}(Mp^2)$ and $\mathcal{O}(Mp+p^2)$,
respectively, where $M$ is the memory size used in these methods.
The other existing decentralized quasi-Newton methods require matrix-vector products and 
need to store quasi-Newton matrices where 
the computation and memory requirements are at least $O(p^2)$.

\end{remark}
\subsubsection{Convergence of DMBFGS}
First, by \eqref{d_quasi}, the Step 3 and Step 4 of Alg.~\ref{alg:Framwork1}
can be rewritten across all the nodes as
\begin{align*} 
	& \m{x}^{t+1}=\m{W} \left(\m{x}^{t}-\alpha \m{H}^t \m{v}^{t}\right), \\ 
	& \m{v}^{t+1}=\m{W} \left(\m{v}^t+\m{g}^{t+1}-\m{g}^t\right),
\end{align*}
where 
\begin{equation}\label{Ht}
	\m{H}^t = \left[ {\begin{array}{*{20}{c}}
			{\m{H}_1^t}
			{}&{}&{}\\
			{}&\ddots&{}\\
			{}&{}&{\m{H}_n^t}
	\end{array}} \right].
\end{equation}
In this section, we would establish the convergence of DMBFGS in a more general setting
such that matrix $\m{H}^t$ defined in \eqref{Ht} satisfies the following assumption.
\begin{assum}\label{as3}
	There exist constants $\Psi > \psi >0$ such that for any $t \ge 0$, the approximation 
matrix $\m{H}^t$ satisfies
	\begin{equation*}
		\psi \m{I} \preceq \m{H}^t \preceq \Psi \m{I}.
	\end{equation*}
\end{assum}
We now show that $\m{H}^t$ given by \eqref{H1}, which is used in Alg.~\ref{alg:Framwork1},
satisfies the Assumption~\ref{as3}. From Step 5 of Alg.~\ref{alg:Framwork1}, 
 if $\m{y}_i^t =\check{\m{y}}_i^t$, we have $l\m{I}_p \preceq \m{H}_i^{t+1} \preceq u\m{I}_p$;
otherwise, we have $\m{y}_i^t=\hat{\m{y}}_i^t$.
Then, it follows from Assumptions~\ref{as0},\ref{as2} that
\[
 (\m{s}_i^t)\tr \hat{\m{y}}_i^t \geq (1/L)\|\hat{\m{y}}_i^t\|^2
 \quad \mbox{ and } \quad   (\m{s}_i^t)\tr \hat{\m{y}}_i^t \geq \mu\|\m{s}_i^t\|^2,
 \]
with together with \eqref{lam_H} gives 
\[
1/(2L) \le \lambda_{\min } (H_i^t(\hat{\m{y}}_i^t)) \le \lambda_{\max} (H_i^t(\hat{\m{y}}_i^t)) 
\le 2/\mu.
\]
Hence, we obtain
\begin{equation*}
	\min \left\{l,\frac{1}{2L}\right\}\m{I} \preceq \m{H}^{t+1} \preceq \max \left\{u,\frac{2}{\mu}\right\}\m{I}.
\end{equation*}
So, the $\m{H}^t$ used in Alg.~\ref{alg:Framwork1} satisfies Assumption~\ref{as3}.

In the following, to show convergence of DMBFGS,
 we define the condition number of the objective function in \eqref{obj_fun1}
and  the condition number of the underlying network $\mathcal{G}$ as
\[
\kappa_f=\frac{L}{\mu}  \quad \mbox{ and } \quad \kappa_g=\frac{1}{1-\sigma^2},
\]
respectively. In general,  $\kappa_g$ measures the network topology and
a smaller $\kappa_g$ implies greater connectivity of the network $\mathcal{G}$.
In addition, we also define
\begin{align*}
	\kappa_H=\frac{\Psi}{\psi} \quad \mbox{ and } \quad
	 \bar{\m{H}}^t=\frac{1}{n} \sum_{i=1}^n \m{H}_i^t.
\end{align*}
We start with establishing some useful lemmas.
\begin{lemma}\label{lm-xxx-aaa}
	Under Assumptions \ref{as0}, \ref{as2}, and \ref{as3}, we have
	\begin{align}\label{xxx-aaa}
		  \|\m{x}^{t+1}-\m{M}\m{x}^{t+1}\|^{2} 
		\leq&\left(\frac{1+\sigma^{2}}{2}+\frac{6\sigma^2\Psi^2\alpha^{2}L^2}{1-\sigma^{2}}\right)\|\m{x}^t-\m{M}\m{x}^t\|^{2}
		+\frac{6\sigma^2\Psi^2\alpha^{2}}{1-\sigma^{2}}\|\m{v}^t-\m{M}\m{v}^t\|^2 \notag\\
		&+\frac{12\sigma^2\Psi^2\alpha^{2}Ln}{1-\sigma^{2}} (F(\bar{\m{x}}^t)-F(\m{z}^*)). 
	\end{align}
\end{lemma}
\begin{proof}
	By the update of $\m{x}^{t+1} = \m{W}(\m{x}^t + \alpha \m{d}^t)$ 
	and Lemma~\ref{property W}, we have
	\begin{align}\label{x_wx}
		&\|\m{x}^{t+1}-\m{M}\m{x}^{t+1}\|^{2}  
		= \|\m{W}\m{x}^t-\m{M}\m{x}^t+\alpha(\m{W}-\m{M})\m{d}^t\|^{2} \notag\\
		=& \|(\m{W}-\m{M})(\m{x}^t-\m{M}\m{x}^t) + \alpha (\m{W}-\m{M})\m{d}^t\|^{2} \notag \\
		\leq&  \frac{1+\sigma^{2}}{2\sigma^{2}}\|(\m{W}-\m{M})(\m{x}^t-\m{M}\m{x}^t)\|^{2}+\frac{(1+\sigma^{2})\alpha^{2}}{1-\sigma^{2}}\|(\m{W}-\m{M})\m{d}^t\|^2 \notag\\
		\leq&\frac{1+\sigma^{2}}{2}\|\m{x}^t-\m{M}\m{x}^t\|^{2}+\frac{2\alpha^{2}\sigma^2}{1-\sigma^{2}}\|\m{d}^t\|^{2},
	\end{align}
where the first inequality applies Young's inequality, Lemma~\ref{young},  with 
$\eta=\frac{1-\sigma^2}{2\sigma^2}$.
Now, it follows from Assumption~\ref{as3} that
	\begin{align*}
		&\|\m{d}^t\|=\|\m{H}^t\m{v}^t\| \leq \Psi \|\m{v}^t\|\\
		= & \Psi \|\m{v}^t-\m{M}\m{v}^t+\m{1}_n \otimes \overline{\nabla} f(\m{x}^t)-\m{1}_n \otimes \nabla F(\bar{\m{x}}^t)+\m{1}_n \otimes \nabla F(\bar{\m{x}}^t)\|\\
		\leq& \Psi\|\m{v}^t-\m{M}\m{v}^t\|+\Psi\sqrt{n}\|\overline{\nabla} f(\m{x}^t)-\nabla F(\bar{\m{x}}^t)\|
		+\Psi\sqrt{n}\|\nabla F(\bar{\m{x}}^t)\|\\
		\leq &\Psi\|\m{v}^t-\m{M}\m{v}^t\|+\Psi L\|{\m{x}}^t-\m{M}{\m{x}}^t\|+\Psi\sqrt{n}\|\nabla F(\bar{\m{x}}^t)\|,
	\end{align*}
where the second equality follows from $\m{1}_n \otimes \overline{\nabla} f(\m{x}^t)=\m{M}\m{g}^t=\m{M}\m{v}^t$ by \eqref{vequalg}, and the last inequality follows from Assumption~\ref{as0} and 
Lemma~\ref{important}.
Taking square on both sides of the above inequality, we get
	\begin{align}\label{I_Wd0}
		\|\m{d}^t\|^2 
		\leq  3\Psi^2\|\m{v}^t-\m{M}\m{v}^t\|^2+3\Psi^2L^2\|{\m{x}}^t-\m{M}\m{x}^t\|^2
		+3\Psi^2n\|\nabla F(\bar{\m{x}}^t)\|^2.
	\end{align}
By the $L$-Lipschitz continuity of $\nabla F$, we have
 $\|\nabla F(\bar{\m{x}}^t)\|^2\leq 2L (F(\bar{\m{x}}^t)-F(\m{z}^*))$.
 Hence, we obtain from \eqref{I_Wd0} that
	\begin{align}\label{I_Wd}
		\|\m{d}^t\|^2 
		\leq  3\Psi^2\|\m{v}^t-\m{M}\m{v}^t\|^2+3\Psi^2L^2\|{\m{x}}^t-\m{M}\m{x}^t\|^2
		+6\Psi^2Ln(F(\bar{\m{x}}^t)-F(\m{z}^*)).
	\end{align}
	Substituting \eqref{I_Wd} into \eqref{x_wx} yields \eqref{xxx-aaa}.
\end{proof}

\begin{lemma}\label{lem 2.6}
	Under Assumptions \ref{as0}, \ref{as2}, and \ref{as3}, we have
	\begin{align}\label{yyy-bbb}
		\|\bar{\m{d}}^t+\bar{\m{H}}^t\nabla F(\bar{\m{x}}^t)\|^2
		\leq  \frac{2\Psi^2}{n}\|\m{v}^t-\m{M}\m{v}^t\|^2+\frac{2\Psi^2L^2}{n}\|{\m{x}}^t-\m{M}\m{x}^t\|^2.
	\end{align}
\end{lemma}
\begin{proof}
	By Lemma~\ref{important}, Assumption~\ref{as0} and Assumption~\ref{as3}, we have
	\begin{align}\label{bars}
		\|\bar{\m{d}}^t+\bar{\m{H}}^t\nabla F(\bar{\m{x}}^t)\|^2
		=&\|\bar{\m{d}}^t+\bar{\m{H}}^t\overline{\nabla} f(\m{x}^t)+\bar{\m{H}}^t(\nabla F(\bar{\m{x}}^t)-\overline{\nabla} f(\m{x}^t))\|^2\\
		\leq & 2\|\bar{\m{d}}^t+\bar{\m{H}}^t\overline{\nabla} f(\m{x}^t)\|^2+2\Psi^2\|\nabla F(\bar{\m{x}}^t)-\overline{\nabla} f(\m{x}^t)\|^2 \notag \\
		\leq & 2\|\bar{\m{d}}^t+\bar{\m{H}}^t\overline{\nabla} f(\m{x}^t)\|^2
	+\frac{2\Psi^2L^2}{n}\|{\m{x}}^t-\m{M}\m{x}^t\|^2. \notag
	\end{align}
Since $\m{d}^t = - \m{H}^t \m{v}^t$, we have
 $\bar{\m{d}}^t = -\frac{1}{n} \sum_{i=1}^n \m{H}_i^t( \m{v}_i^t-\overline{\nabla} f(\m{x}^t))-\bar{\m{H}}^t\overline{\nabla} f(\m{x}^t)$.
 Then, it follows from Lemma~\ref{important} and Assumption \ref{as3} that
	\begin{align*}
		\|\bar{\m{d}}^t+\bar{\m{H}}^t\overline{\nabla} f(\m{x}^t)\|^2
		=\|\frac{1}{n} \sum_{i=1}^n\m{H}_i^t( \m{v}_i^t-\overline{\nabla} f(\m{x}^t)) \|^2
		\leq  \frac{\Psi^2}{n}\|\m{v}^t-\m{M}\m{v}^t\|^2.
	\end{align*}
The above inequality and \eqref{bars} give \eqref{yyy-bbb}.
\end{proof}

\begin{lemma}\label{lm-xyz-abc}
	Under Assumptions \ref{as0}, \ref{as2}, and \ref{as3}, we have
	\begin{align}\label{xyz-abc}
		n\left(F(\bar{\m{x}}^{t+1})-F(\m{z}^*) \right) 
		\leq &\left[1-\left(\psi\alpha-3\alpha^2L\Psi^2 \right)\mu\right] n(F(\bar{\m{x}}^t)-F(\m{z}^*))\\
		&+ \left(\frac{\alpha L^2\Psi^2}{\psi}+\frac{3L^3\Psi^2\alpha^{2}}{2} \right)\|{\m{x}}^t-\m{M}\m{x}^t\|^2\notag\\
		&+ \left(\frac{\alpha\Psi^2}{\psi}+\frac{3L\Psi^2\alpha^{2}}{2} \right)\|\m{v}^t-\m{M}\m{v}^t\|^2.\notag
	\end{align}
\end{lemma}
\begin{proof}
	Taking the average of the update $\m{x}^{t+1}  = \m{W}(\m{x}^t + \alpha \m{d}^t)$, we have
	\begin{equation*}
		\bar{\m{x}}^{t+1}=\bar{\m{x}}^{t}+\alpha \bar{\m{d}}^{t}.
	\end{equation*}
	Then, by the $L$-Lipschitz continuity of $\nabla F$, it holds that
	\begin{align}\label{fbarx}
		F(\bar{\m{x}}^{t+1})\leq F(\bar{\m{x}}^t)+\alpha\Big\langle\nabla F(\bar{\m{x}}^t),\bar{\m{d}}^t\Big\rangle+\frac{L\alpha^{2}}{2}\|\bar{\m{d}}^t\|^{2}.
	\end{align}
	From Assumption~\ref{as3}, we have
	 $\left\langle \nabla F(\bar{\m{x}}^t), \bar{\m{H}}^t \nabla F(\bar{\m{x}}^t)\right\rangle \geq \psi \|\nabla F(\bar{\m{x}}^t)\|^2 $. Then, it follows from the Young's inequality, Lemma~\ref{young},
	 with  $\eta=\psi$ that
	\begin{align}\label{fd}
		\Big\langle\nabla F(\bar{\m{x}}^t),\bar{\m{d}}^t\Big\rangle
		\leq &-\psi\|\nabla F(\bar{\m{x}}^t)\|^2+\Big\langle\nabla F(\bar{\m{x}}^t),\bar{\m{d}}^t+\bar{\m{H}}^t\nabla F(\bar{\m{x}}^t)\Big\rangle\\
		\leq&-\frac{\psi}{2}\|\nabla F(\bar{\m{x}}^t)\|^2+\frac{1}{2\psi}\|\bar{\m{d}}^t+\bar{\m{H}}^t\nabla F(\bar{\m{x}}^t)\|^2\notag\\
		\leq &-\frac{\psi}{2}\|\nabla F(\bar{\m{x}}^t)\|^2+\frac{\Psi^2}{\psi n}\|\m{v}^t-\m{M}\m{v}^t\|^2
		+\frac{\Psi^2L^2}{\psi n}\|{\m{x}}^t-\m{M}\m{x}^t\|^2,\notag
	\end{align}
	where the last inequality is by Lemma~\ref{lem 2.6}. 
	Moreover, by Lemma~\ref{important} and \eqref{I_Wd0}, we have
	\begin{align}\label{Ld}
		\frac{L\alpha^{2}}{2}\|\bar{\m{d}}^t\|^{2}\leq\frac{L\alpha^{2}}{2n}\|{\m{d}}^t\|^{2}
		\leq& \frac{3L\Psi^2\alpha^{2}}{2n}\|\m{v}^t-\m{M}\m{v}^t\|^2+\frac{3L^3\Psi^2\alpha^{2}}{2n}\|{\m{x}}^t-\m{M}\m{x}^t\|^2\\
		&+\frac{3}{2}L\Psi^2\alpha^{2}\|\nabla F(\bar{\m{x}}^t)\|^2.\notag
	\end{align}
	Substituting \eqref{fd} and \eqref{Ld} into \eqref{fbarx} yields
	\begin{align}\label{nFF}
		nF(\bar{\m{x}}^{t+1})\leq& nF(\bar{\m{x}}^t)-\left( \frac{\psi\alpha}{2}-\frac{3\alpha^2L\Psi^2}{2}\right) n\|\nabla F(\bar{\m{x}}^t)\|^2\\
		&+\left( \frac{\alpha L^2\Psi^2}{\psi}+\frac{3L^3\Psi^2\alpha^{2}}{2}\right) \|{\m{x}}^t-\m{M}\m{x}^t\|^2
		+\left( \frac{\alpha\Psi^2}{\psi}+\frac{3L\Psi^2\alpha^{2}}{2}\right) \|\m{v}^t-\m{M}\m{v}^t\|^2\notag.
	\end{align}
   From the  $\mu$-strong convexity of $F$, we have $\|\nabla F(\bar{\m{x}}^t)\|^2\geq 2\mu (F(\bar{\m{x}}^t)-F(\m{z}^*))$.
   Then, we have from \eqref{nFF} that
	\begin{align*}
		nF(\bar{\m{x}}^{t+1})\leq &nF(\bar{\m{x}}^t)-\left( \psi\alpha-3\alpha^2L\Psi^2\right) \mu n(F(\bar{\m{x}}^t)-F(\m{z}^*))\\
		&	+\left( \frac{\alpha L^2\Psi^2}{\psi}+\frac{3L^3\Psi^2\alpha^{2}}{2}\right) \|{\m{x}}^t-\m{M}\m{x}^t\|^2
		+\left( \frac{\alpha\Psi^2}{\psi}+\frac{3L\Psi^2\alpha^{2}}{2}\right) \|\m{v}^t-\m{M}\m{v}^t\|^2.
	\end{align*}
	Subtracting $nF(\m{z}^*)$ on both sides of the above inequality gives \eqref{xyz-abc}.
\end{proof}
\begin{lemma}\label{lem2.9}
	Under Assumptions \ref{as0}, \ref{as2}, and \ref{as3}, we have
	\begin{align}\label{zzz-ccc}
		\|\m{x}^{t+1}-\m{x}^t\|^2
		\leq&\left(8+6\alpha^2L^2\Psi^2\right)\|\m{x}^t-\m{M}\m{x}^t\|^{2}+6\alpha^2\Psi^2\|\m{v}^t-\m{M}\m{v}^t\|^2 \notag \\
		&+12\alpha^2L\Psi^2n(F(\bar{\m{x}}^t)-F(\m{z}^*)).
	\end{align}
\end{lemma}
\begin{proof}
    It follows from $\|(\m{W}-\m{I}) \m{x}^t\|=\|(\m{W}-\m{I}) (\m{x}^t-\m{M}\m{x}^t)\| \leq 2 \|\m{x}^t-\m{M}\m{x}^t\|$ that
	\begin{align*}
		& \|\m{x}^{t+1}-\m{x}^t\|^2
		=\|(\m{W}-\m{I}) \m{x}^t+\alpha\m{W} \m{d}^t\|^2 \leq 8\|\m{x}^t-\m{M}\m{x}^t\|^2+2\alpha^2\|\m{d}^t\|^2.
	\end{align*}
	Then, the above inequality and \eqref{I_Wd} give \eqref{zzz-ccc}.
\end{proof}

\begin{lemma}\label{lm-www-ddd}
	Under Assumptions \ref{as0}, \ref{as2}, and \ref{as3}, we have
	\begin{align}\label{www-ddd}
		\|\m{v}^{t+1}-\m{M}\m{v}^{t+1}\|^2
		\leq
		&\left( \frac{12\alpha^2L^2\Psi^2\sigma^2}{1-\sigma^2}+\frac{1+\sigma^2}{2}\right) \|\m{v}^t-\m{M}\m{v}^t\|^2\\
		&+\frac{2L^2\sigma^2}{1-\sigma^2}\left(8+6\alpha^2L^2\Psi^2\right)\|\m{x}^t-\m{M}\m{x}^t\|^{2}\notag\\
		&+\frac{24\alpha^2L^3\Psi^2 \sigma^2}{1-\sigma^2}n(F(\bar{\m{x}}^t)-F(\m{z}^*))\notag.
	\end{align}
\end{lemma}
\begin{proof}
	By the update of $\m{v}^{t+1}$ and the Young's inequality, Lemma~\ref{young},
	 with $\eta=\frac{1-\sigma^2}{2\sigma^2}$, we have 
	\begin{align*}
		&\|\m{v}^{t+1}-\m{M}\m{v}^{t+1}\|^2=\| \m{W}{\m{v}}^{t}+\m{W}\m{g}^{t+1}-\m{W}\m{g}^{t} -\m{M}\m{v}^t-\m{M}\m{g}^{t+1}+\m{M}\m{g}^t\|^2\\
		\leq&\frac{1+\sigma^2}{2\sigma^2}\|\m{W}\m{v}^t-\m{M}\m{v}^t\|^2+\frac{1+\sigma^2}{1-\sigma^2}\|(\m{W}-\m{M})(\m{g}^{t+1}-\m{g}^t)\|^2\\
		\leq&\frac{1+\sigma^2}{2}\|\m{v}^t-\m{M}\m{v}^t\|^2+\frac{2\sigma^2}{1-\sigma^2}\|\m{g}^{t+1}-\m{g}^t\|^2\\
		\leq& \frac{1+\sigma^2}{2}\|\m{v}^t-\m{M}\m{v}^t\|^2+\frac{2L^2\sigma^2}{1-\sigma^2}\|\m{x}^{t+1}-\m{x}^t\|^2\\
		\leq &\frac{2L^2\sigma^2}{1-\sigma^2}\left(8+6\alpha^2L^2\Psi^2\right)\|\m{x}^t-\m{M}\m{x}^t\|^{2}
		+\left( \frac{12\alpha^2L^2\Psi^2\sigma^2}{1-\sigma^2}+\frac{1+\sigma^2}{2}\right) \|\m{v}^t-\m{M}\m{v}^t\|^2\\
		&+\frac{24\alpha^2L^3\Psi^2 \sigma^2}{1-\sigma^2}n(F(\bar{\m{x}}^t)-F(\m{z}^*)),
	\end{align*}
	where we use Lemma~\ref{property W} in the second inequality and
	the $L$-Lipschitz continuity of $\nabla f$ in the third inequality.
Then, the above inequality and Lemma~\ref{lem2.9} give \eqref{www-ddd}.
\end{proof}

So far, we have established the bounds on the consensus error $\|\m{x}^{t+1}-\m{M}\m{x}^{t+1}\|^{2} $,  the objective function value gap $F(\bar{\m{x}}^{t+1})-F(\m{z}^*)$, and the gradient tracking error $\|\m{v}^{t+1}-\m{M}\m{v}^{t+1}\|^2 $. We now stack the three  errors together
and define the total error vector
\begin{equation}\label{def-ut}
	\m{u}^{t+1} =\left[\begin{array}{c}
		\|\m{x}^{t+1}-\m{M}\m{x}^{t+1}\|^{2}   \\
		n\left(F(\bar{\m{x}}^{t+1})-F(\m{z}^*) \right)  \\
		\|\m{v}^{t+1}-\m{M}\m{v}^{t+1}\|^2 
	\end{array}\right] .
\end{equation}
Then, $\m{u}^t$ measures the error between $\m{x}_i^t$ and $\m{z}^*$,
 since $\m{u}^t=\m{0}$ implies that $\m{x}_i^t=\bar{\m{x}}^t=\m{z}^*$ for all $i = 1, \ldots, n$.
Now the following theorem shows under proper choice of stepsize $\alpha$, 
the error $\|\m{u}\|^t$ will goes to zero linearly.

\begin{theorem}\label{DMBFGS-Thm}
	Under Assumptions \ref{as0}, \ref{as2}, and \ref{as3}, we have
	\begin{equation}\label{u-retation}
		\m{u}^{t+1} \le \m{J} \m{u}^t,
	\end{equation}
	where
	\begin{align}\label{matrix}
		\m{J} = 
		\Bigg[ {\begin{array}{*{20}{c}}
				{\frac{1+\sigma^{2}}{2}+\frac{6\sigma^2\Psi^2\alpha^{2}L^2}{1-\sigma^{2}}
				}&{\frac{12\sigma^2\Psi^2\alpha^{2}L}{1-\sigma^{2}}}&{\frac{6\sigma^2\Psi^2\alpha^{2}}{1-\sigma^{2}}}\\
				{\frac{\alpha L^2\Psi^2}{\psi}+\frac{3L^3\Psi^2\alpha^{2}}{2}}&{1-(\psi\alpha-3\alpha^2L\Psi^2)\mu}&{\frac{\alpha\Psi^2}{\psi}+\frac{3L\Psi^2\alpha^{2}}{2}}\\
				{\frac{2L^2\sigma^2}{1-\sigma^2}\left(8+6\alpha^2L^2\Psi^2\right)}&{\frac{24\alpha^2L^3\Psi^2 \sigma^2}{1-\sigma^2}}&{ \frac{12\alpha^2L^2\Psi^2\sigma^2}{1-\sigma^2}+\frac{1+\sigma^2}{2}}
		\end{array}}\Bigg].
	\end{align}
	If the stepsize $\alpha$ satisfies
	\begin{align} \label{alpha-0}
	 0<	\alpha \leq \sqrt{\frac{1}{3916}}\frac{(1-\sigma^2)^2}{ L\Psi \kappa_H }\sqrt{\frac{1}{\kappa_f}}.
	\end{align}
	then the spectral radius $\rho(\m{J})$ of $\m{J}$ will satisfy
	\begin{equation}\label{rho-J}
	 \rho(\m{J}) \le 1- \frac{1}{3916} \frac{(1-\sigma^2)^2}{\kappa_f^2 \kappa_H^2},
     \end{equation}   	 
	 and $\sum_{i}^n \|\m{x}_i^t-\m{z}^*\|$ converges to zero linearly with rate
	 $\rho(\m{J}) \in (0, 1)$.\\
    If, in particular, the network $\mathcal{G}$ is fully connected, i.e., $\sigma=0$,
    and the stepsize $\alpha$ satisfies
	\begin{equation} \label{alpha-1}
		0< \alpha \leq \frac{1}{15 L \Psi \kappa_H},
	\end{equation}
   we have $\rho(J) \le 1-\frac{1}{225 \kappa_H^2\kappa_f}$.
\end{theorem}
\begin{proof}
    First, \eqref{u-retation} follows directly from Lemma~\ref{lm-xxx-aaa},  Lemma~\ref{lm-xyz-abc}
    and Lemma~\ref{lm-www-ddd}.

	We now find a range of $\alpha$ such that $\rho(\m{J})<1$. 
    By Lemma~\ref{important1}, if we can find a positive vector $\m{v} = [v_1;v_2;v_3] \in \R^3$	
    such that 
\begin{equation}\label{J-eq-1}    
    \m{J}\m{v}\leq(1-\theta\Psi^2\alpha^2)\m{v}
\end{equation}    
     for some constant $\theta >0$, then we will have $\rho(\m{J})\le 1-\theta\Psi^2\alpha^2 < 1$.
	Note that inequalities \eqref{J-eq-1} can be equivalently written as
	\begin{align*}
		\left\{\begin{array}{cl}
			&\frac{12\sigma^2\Psi^2\alpha^{2}L}{1-\sigma^{2}}v_2+\frac{6\sigma^2\Psi^2\alpha^{2}}{1-\sigma^{2}}v_3
			\leq\left(1-\theta\Psi^2\alpha^2-\frac{1+\sigma^{2}}{2}-\frac{6\sigma^2\Psi^2\alpha^{2}}{1-\sigma^{2}}\right)v_1,\\
			&\left( \frac{\alpha L^2\Psi^2}{\psi}+\frac{3L^3\Psi^2\alpha^{2}}{2}\right) v_1+\left( \frac{\alpha\Psi^2}{\psi}+\frac{3L\Psi^2\alpha^{2}}{2}\right) v_3
			\leq\left( \left( \psi\alpha- 3\alpha^2L\Psi^2\right) \mu-\theta\Psi^2\alpha^2\right)  v_2,\\
			&\frac{2L^2\sigma^2}{1-\sigma^2}\left(8+6\alpha^2L^2\Psi^2\right)v_1+\frac{24\alpha^2L^3\Psi^2 \sigma^2}{1-\sigma^2}v_2
			\leq \left(\frac{1-\sigma^2}{2}-\theta\Psi^2\alpha^2- \frac{12\alpha^2L^2\Psi^2\sigma^2}{1-\sigma^2}\right)v_3,
		\end{array}\right.
	\end{align*}
which can be reorganized as
	\begin{align}\label{system}
	\left\{\begin{array}{cl}
		&\left(\frac{12\sigma^2\Psi^2L}{1-\sigma^{2}}v_2+\frac{6\sigma^2\Psi^2}{1-\sigma^{2}}v_3+\theta\Psi^2 v_1+\frac{6\sigma^2\Psi^2}{1-\sigma^{2}}v_1\right)\alpha^2 \leq \frac{1-\sigma^{2}}{2}v_1,\\
		&\left(\frac{3L^3\Psi^2}{2}v_1+\frac{3L\Psi^2}{2}v_3+3\mu L\Psi^2 v_2+\theta\Psi^2 v_2\right)\alpha^2 \leq \left(\psi\mu v_2-\frac{ L^2\Psi^2}{\psi}v_2- \frac{\Psi^2}{\psi}v_3\right)\alpha,\\
		&\left(\frac{12L^4\Psi^2\sigma^2}{1-\sigma^2}v_1+\frac{24L^3\Psi^2 \sigma^2}{1-\sigma^2}v_2+\theta\Psi^2v_3+\frac{12L^2\Psi^2\sigma^2}{1-\sigma^2}v_3\right)\alpha^2 \leq  \frac{1-\sigma^{2}}{2}v_3-\frac{16L^2\sigma^2}{1-\sigma^2}v_1.
	\end{array}\right.
\end{align}
	Let us choose vector $\m{v} = [v_1;v_2;v_3] > \m{0}$ such that 
		\begin{equation}\label{v-system}
		\left\{\begin{array}{cl}
			\mu v_2- L^2 \kappa_H^2 v_1-\kappa_H^2v_3>0, \\
			(1-\sigma^2)^2v_3-16L^2\sigma^2v_1>0.
		\end{array}\right.
	\end{equation}
	Then, we can see that the inequality system \eqref{system} will hold for all 
	 sufficiently small $\alpha >0$. 
	 In particular, by taking $v_1=\frac{(1-\sigma^2)^2}{32L^2\sigma^2}v_3$ and
	 $v_2=\frac{\kappa_H^2}{\mu}\left(\frac{(1-\sigma^2)^2}{32\sigma^2}+2\right)v_3$ for some 
	 $v_3 >0$, we will have \eqref{v-system} hold and
	  the inequality system \eqref{system} holds as long as 
	\begin{align}\label{alpha-range}
0< \alpha \leq \min \Bigg\{ &\frac{(1-\sigma^2)^{2}}{L\Psi \kappa_H \sigma}\sqrt{\frac{1}{978+\theta\frac{(1-\sigma^2)^3\mu}{L^3\kappa_H^2\sigma^2}}}\sqrt{\frac{\mu}{L}}, \; \frac{64}{L\Psi \kappa_H \left(489+\frac{130\theta}{\mu L}\right)}, \nonumber \\
		&\frac{2(1-\sigma^2)^{1.5}}{L\Psi \kappa_H \sigma}\sqrt{\frac{1}{489+\theta\frac{8(1-\sigma^2)\mu}{L^3\kappa_H^2\sigma^2}}}\sqrt{\frac{\mu}{L}}\Bigg\}.
	\end{align}
By setting $\theta=\frac{\mu L}{(1-\sigma^2)^2}$ and some simple calculations, we will see that \eqref{alpha-range} will be satisfied as long as $\alpha$ satisfies condition \eqref{alpha-0},
which ensures $\rho(\m{J}) \le 1-\theta\Psi^2 \alpha^2$ and therefore, \eqref{rho-J} holds. 
Then, by the definition of $\m{u}^t$ in \eqref{def-ut} and  \eqref{rho-J}, we have
\begin{align*}
	& \sum_{i=1}^n \|\m{x}_i^t-{\m{z}}^*\|^2 = \|\m{x}^t-\tilde{\m{z}}^*\|^2 \\
\leq&\|\m{x}^t-\m{M}\m{x}^t\|^2+\|\m{M}\m{x}^t-\tilde{\m{z}}^*\|^2
	= \|\m{x}^t-\m{M}\m{x}^t\|^2+n\|\bar{\m{x}}^t-{\m{z}}^*\|^2\\
	\leq&\|\m{x}^t-\m{M}\m{x}^t\|^2+\frac{2 }{\mu}n\left(F(\bar{\m{x}}^{t})-F(\m{z}^*)\right)\\
	\le & \max \{1, \sqrt{2}/\mu \} \|\m{u}^t\|_{\infty} \le \max \{1, \sqrt{2}/\mu \}  
	\|\m{u}^0\| \rho(\m{J})^t,
\end{align*}
where $\tilde{\m{z}}^*=[\m{z}^*;\ldots;\m{z}^*] \in \R^{np}$;
 the second inequality is due to the $\mu$-strong convexity of $F$ 
 and the last inequality follows from \eqref{u-retation}.
Hence, $\sum_{i}^n \|\m{x}_i^t-\m{z}^*\|$ converges to zero linearly with rate
	 $\rho(\m{J}) \in (0, 1)$.

Now, if the network is fully connected, i.e., $\sigma=0$ , the inequality \eqref{system} can be simplified as
	 \begin{align}\label{system1}
	 		\left\{\begin{array}{cl}
	 		&\theta\Psi^2\alpha^2 \leq \frac{1}{2},\\
	 		&\left( \frac{\alpha L^2\Psi^2}{\psi}+\frac{3L^3\Psi^2\alpha^{2}}{2}\right) v_1+\left( \frac{\alpha\Psi^2}{\psi}+\frac{3L\Psi^2\alpha^{2}}{2}\right) v_3\\
	 		&\leq\left( \left( \psi\alpha-3\alpha^2L\Psi^2 \right) \mu-\theta\Psi^2\alpha^2\right)  v_2.
	 	\end{array}\right.
	 \end{align}
Hence, if we  choose vector $\m{v} = [v_1;v_2;v_3] > \m{0}$ such that 	 
 \begin{equation}\label{v-system1}
 		\mu v_2-L^2 \kappa_H^2 v_1- \kappa_H^2v_3>0,
 \end{equation}
the inequality system \eqref{system1} will hold for all sufficiently small $\alpha >0$.
In particular,  by  taking  $v_1=v_3/L^2$ and $v_2= 3 \kappa_H^2 v_3 /\mu$ for some $v_3 >0$, and setting  $\theta=\mu L$
we will have  \eqref{v-system1} hold and the inequality system \eqref{system1} holds
as long as $\alpha$ satisfies \eqref{alpha-1}, which ensures $\rho(\m{J}) \le 1-\theta\Psi^2 \alpha^2 \le 1-\frac{1}{225 \kappa_H^2\kappa_f}$.
\end{proof}

\begin{remark} We have the following comments for DMBFGS.
\begin{itemize}
\item[(a)]	To reach $\epsilon$-accuracy, that is, $\sum_{i}^n \|\m{x}_i^t-\m{z}^*\| \le \epsilon$
for some $\epsilon >0$, by Theorem~\ref{DMBFGS-Thm},
the number of iterations needed by DMBFGS is $\mathcal{O}\left(\frac{\kappa_f^2 \kappa_H^2 }{(1-\sigma^2)^2}log\left(\frac{1}{\epsilon}\right)\right)$ and will be improved to
 $\mathcal{O}\left(\kappa_f\kappa_H^2log\left(\frac{1}{\epsilon}\right)\right)$ when the network is fully connected, i.e., $\sigma =0$.
 On the other hand, the iterative complexity of D-LM-BFGS \cite{zhang2023variance} is
  always $\mathcal{O}\left(\frac{\kappa_f^2 \kappa_H^2 }{(1-\sigma^2)^2}log\left(\frac{1}{\epsilon}\right)\right)$.
\item[(b)]	
	Considering the case that $\m{H}^t=\m{I}$, i.e., $\kappa_H=1$, DMBFGS will reduce to
	some ATC GT method in \cite{nedic2017geometrically},
	which is shown to have  the iterative complexity $\mathcal{O}\left(\frac{\kappa_f^{1.5} n^{0.5} }{(1-\sigma)^2}log\left(\frac{1}{\epsilon}\right)\right)$.
    In this case, we show the iteration complexity of DMBFGS is  
    $\mathcal{O}\left(\frac{\kappa_f^2 }{(1-\sigma^2)^2}log\left(\frac{1}{\epsilon}\right)\right)$,
    which does not depend on the number of nodes $n$ while is more sensitive to
     the condition number $\kappa_f$. 
   On the other hand, for fully connected network,  both \cite{nedic2017geometrically} 
   and this paper give the iteration complexity $\mathcal{O}\left(\kappa_flog\left(\frac{1}{\epsilon}\right)\right)$, which matches that of the centralized steepest descent method for smooth and
   strongly convex optimization.
	\end{itemize}
\end{remark}

\section{NUMERICAL EXPERIMENTS}
In this section, we would like to test and compare our developed NDCG and DMBFGS with 
some well-developed first-order algorithms 
on solving nonconvex and strongly convex optimization problems, respectively, over a connected 
undirected network with edge density $d \in (0,1]$.
For the generated network $\mathcal{G}=\left(\mathcal{V},\mathcal{E}\right)$, we choose the Metropolis constant edge weight matrix \cite{xiao2007distributed} as the mixing matrix, that is
\begin{equation*}
	\tilde{W}_{i j}=\left\{\begin{array}{cl}
		\frac{1}{\max \{\operatorname{deg}(i), \operatorname{deg}(j)\}+1}, & \text { if }(i, j) \in \mathcal{E}, \\
		0, & \text { if }(i, j) \notin \mathcal{E} \text { and } i \neq j, \\
		1-\sum_{k \in \mathcal{N}_i/ \{i\}} \tilde{W}_{i k}, & \text { if } i=j,
	\end{array}\right.
\end{equation*}
where $(i, j) \in \mathcal{E}$ indicates there is an edge between 
node $i$ and node $j$, and $\operatorname{deg}(i)$ means the degree of node $i$. 
The communication volume in our experiments is calculated as
\begin{align*}\notag
	\text{Communication volume} 
       = &\text{~number of iterations} \\
	&\times \text{number of communication rounds per iteration}\\
	&\times \text{number of edges, i.e., }dn(n-1)/2 \\
	& \times \text{dimension of transmitted vectors on each edge}.
\end{align*}
In all experiments, we set the number of nodes $n=10$ and the edge density $d=0.56$ for the network,
and use initial point $\m{x}^0=\m{0}$ for all comparison algorithms.
All experiments are coded in MATLAB R2017b and run on a laptop with Intel Core i5-9300H CPU, 16GB RAM, and Windows 10 operating system.

\subsection{Experiments for NDCG}
We consider the nonconvex decentralized binary classification problem, which uses
 a logistic regression formulation with a nonconvex regularization as
\begin{equation}\label{noncovex_logistic_problem}
	\mathop {\min }\limits_{\m{z} \in {\R^p}} \sum_{i=1}^n  \sum_{j=1}^{n_i} \log \left(1+\exp (-b_{ij} \m{a}_{ij}\tr \m{z} ) \right)+\hat{\lambda} \sum_{k=1}^p \frac{\m{z}_{[k]}^2}{1+\m{z}_{[k]}^2},
\end{equation}
where $\m{a}_{ij} \in \R^p$ are the feature vectors, $ b_{ij} \in \{-1,+1\}$ are the labels, $\m{z}_{[k]}$ denotes the $k$-th component of the vector $\m{z}$, and $\hat{\lambda}>0$ is the regularization parameter.
The experiments are conducted with regularization parameter $\hat{\lambda}=1$ on four datasets in Table~\ref{table2}
from the LIBSVM library: \textbf{mushroom}, \textbf{ijcnn1}, \textbf{w8a} and \textbf{a9a}.
We would compare NDCG, i.e., Alg.~\ref{alg:Framwork+}, with the following algorithms:
some GT-type method \cite{lu2019gnsd} (GT for short in Section 3.1), Global Update Tracking ({GUT}) \cite{NEURIPS2023_98f8c89a}, {MT} \cite{takezawa2022momentum} and {DSMT} \cite{huang2024accelerated}.
Although {GUT}, {MT}, and {DSMT} are stochastic methods, since we are focusing on comparing 
deterministic decentralized methods, full gradient is used for these methods. 
From the first-order optimality condition  \eqref{stationarity}, 
we define the optimality error at $\m{x}^t$ as 
\begin{equation*}
	\mbox{Optimality error} := \left\| \frac{1}{n}\sum_{i=1}^n\nabla f_i(\m{x}^t_{i}) \right\|+\|\m{x}^t-\m{M}\m{x}^t\|,
\end{equation*}
where $t$ is the iteration number.
\begin{table}[H]
	\caption{Datasets}\label{table2}
	\centering
	\begin{tabular}{c|c|c}
		\hline
		{Dataset}&\# of samples ($\sum_{i=1}^n n_i$)  &\# of features ($p$)  \\
		\hline
		\textbf{mushroom}&8120 &112\\
		\hline
		\textbf{ijcnn1}&49990&22\\
		\hline
		\textbf{w8a}&49740&300\\
		\hline
		\textbf{a9a}&32560&123\\
		\hline
	\end{tabular}
\end{table}
All algorithm parameters are set according to their best performance for the data set
 \textbf{mushroom} (\textbf{ijcnn1}; \textbf{w8a}; \textbf{a9a}). 
In particular, using the same notation in their source papers, we  
set $\eta=0.06 (0.09; 0.09; 0.08)$ for GT, $\eta_t=0.03\sqrt{\frac{n}{t}}$ and $\mu=0.3 (0.3; 0.3; 0.2)$ for GUT, $\eta=0.05$ and $\beta=0.31 (0.41; 0.33; 0.35)$ for MT, $\eta_{w}=1/(1+\sqrt{1-\sigma^2})$, $\beta=1-(1-\sqrt{\eta_{w}})/n^{1/3}$ and $\alpha=0.04 (0.08; 0.04; 0.08)$ for DSMT, where $\sigma$ is
defined in Lemma~\ref{property W}, and $\alpha=0.08 (0.11; 0.1; 0.1)$ for NDCG.

All algorithms except GUT need two rounds of communication per iteration. GUT needs only one round communication per iteration but uses a decreasing stepsize, which yields slow convergence as shown in
Fig.~\ref{noncovex}. Hence, only optimality error against communication volume is shown in 
Fig.~\ref{noncovex}, where we can see NDCG is significantly better than GT and
 outperforms the momentum-based methods MT and DSMT for this set of nonconvex classification problems. 
%
%

\begin{figure*}[!t]
	\centering
	\subfloat[\textbf{mushroom}]{\includegraphics[width=2.5in]{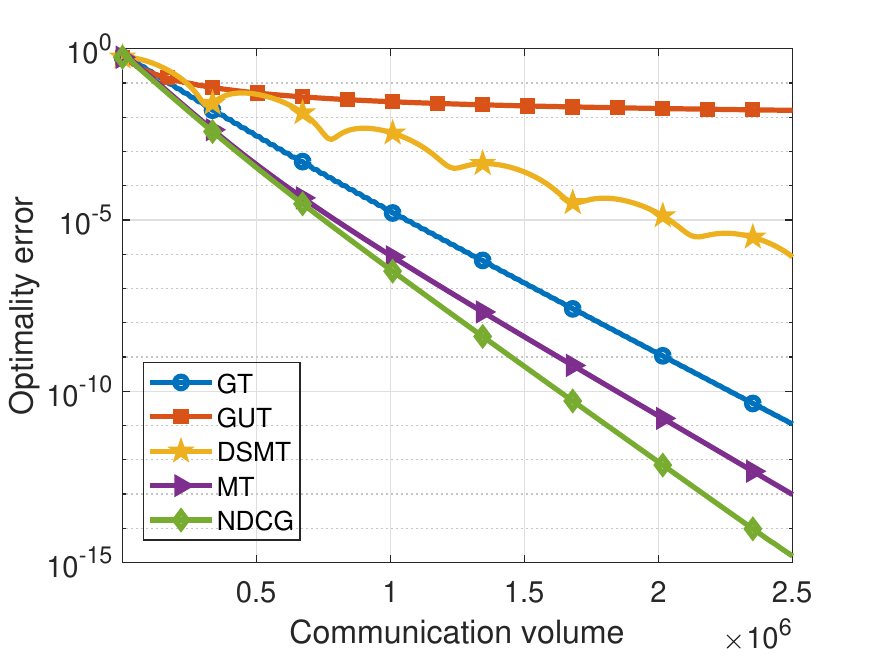}%
		\label{fig_n_a}}
	\subfloat[\textbf{ijcnn1}]{\includegraphics[width=2.5in]{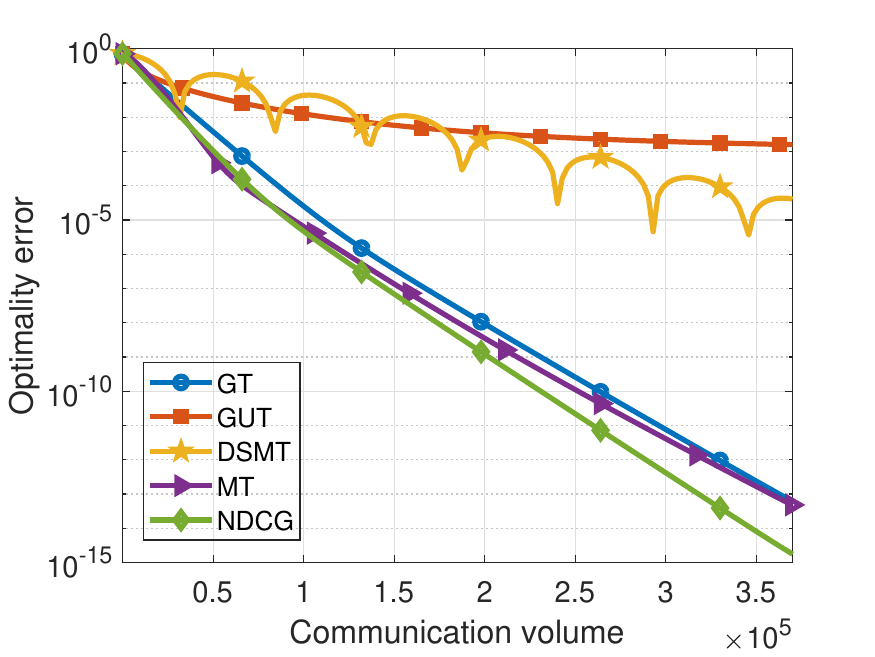}%
		\label{fig_n_b}}
	\hfil
	\subfloat[\textbf{w8a}]{\includegraphics[width=2.5in]{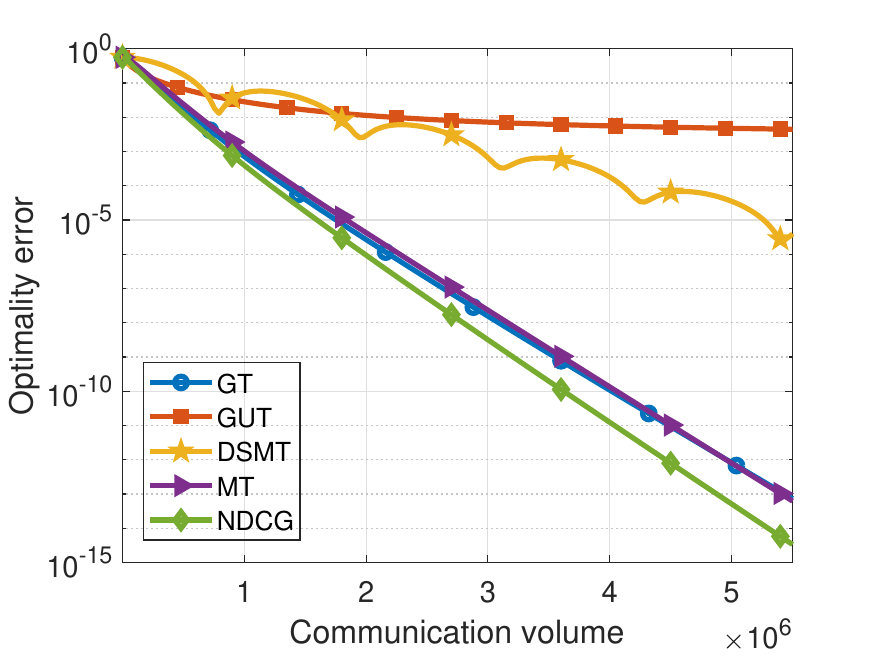}%
		\label{fig_n_c}}
	\subfloat[\textbf{a9a}]{\includegraphics[width=2.5in]{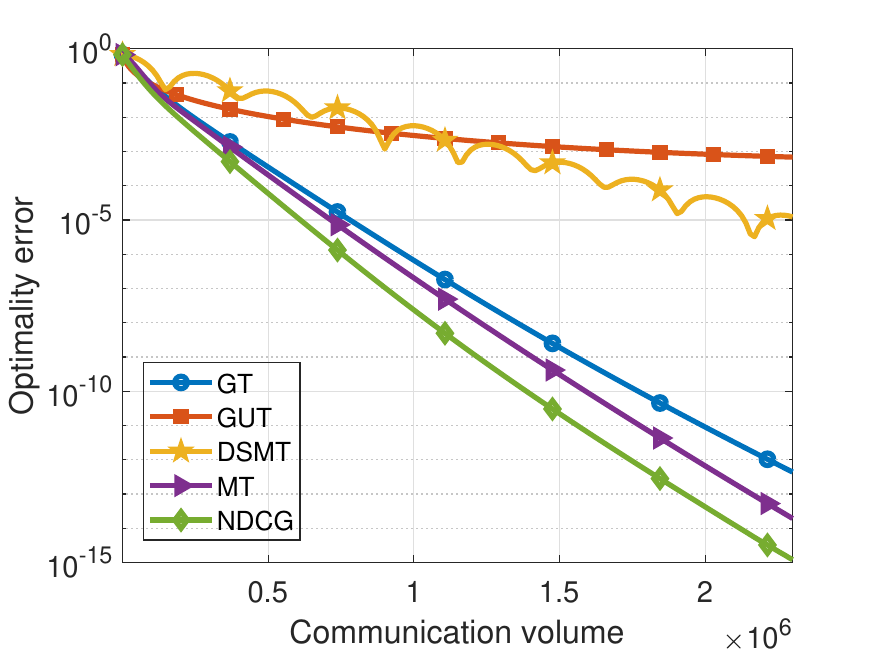}%
		\label{fig_n_d}}
	\caption{Optimality error of comparison algorithms for minimizing the nonconvex logistic regression problem \eqref{noncovex_logistic_problem} on different datasets.}
	\label{noncovex}
\end{figure*}

\subsection{Experiments for DMBFGS}
In this subsection, we choose strongly convex smooth test problems
including the linear regression problem and 
the logistic regression problem with a strongly convex regularization. 
We would compare {DMBFGS}, i.e., Alg.~\ref{alg:Framwork1}, with the following
algorithms: 
\begin{itemize}
	\item Gradient-based methods including some GT-type method \cite{qu2017harnessing} (GT for short in Section 3.2), {ABm} \cite{xin2019distributed}, and {OGT} \cite{song2024optimal};
	\item Quasi-Newton methods including DR-LM-DFP \cite{zhang2023variance} and D-LM-BFGS \cite{zhang2023variance}.
\end{itemize}
Since randomization techniques are used in OGT, we would run OGT $5$ times and record the best
and average performance as OGT(BEST) and OGT(AVG), respectively.
By the strong convexity, there is only one unique optimal solution $\m{z}^*$
for each testing problem. Hence, the optimal error at $\m{x}^t$ is measured as
\begin{equation}\label{rel_error}
	\mbox{Relative error} :=	\frac{1}{n}\sum_{i=1}^n\frac{\|\m{x}^t_i-\m{z}^*\|}{\|\m{z}^*\|+1},
\end{equation}
where the true solution $\m{z}^*$ is explicitly obtained for the linear regression problem
and is pre-computed for the logistic regression problem.

\subsubsection{Linear regression problem}
We first use linear regression problem to test how the objective function condition number $\kappa_f$ 
would affect the algorithm's performance. We compare DMBFGS with 
 gradient-based algorithms: {GT} \cite{qu2017harnessing}, {ABm} \cite{xin2019distributed}, and {OGT} \cite{song2024optimal}. In particular, we consider
\begin{equation}\label{linear_problem}
	\mathop {\min }\limits_{\m{z} \in {\R^p}} \sum_{i=1}^n \frac{1}{2} \m{z}\tr \m{A}_i \m{z} + \m{b}_i\tr \m{z},
\end{equation}
where $\m{A}_i \in \R^{p \times p}$ and $\m{b}_i \in \R^p$ are private data available to node $i$. To control the condition number of problem \eqref{linear_problem}, we generate $\m{A}_i = \m{Q}\tr \operatorname {diag} \{a_1,...,a_p \} \m{Q}$, where $\m{Q}$ is a random orthogonal matrix. We set $a_1=1$ and $a_p$ as an arbitrarily large number, 
and generate $a_j \sim \operatorname{U}(1,2)$ for $j=2, \ldots, p-1$, where $\operatorname{U}(1,2)$ represents the uniform distribution from 1 to 2. So, $\kappa_f=a_p/a_1=a_p$.
In numerical experiments, we set $p=1000$ and $a_p$ = $10^2$, $10^3$, $10^4$, $10^5$, respectively. 

Again, for $\kappa_f=10^2 (10^3; 10^4; 10^5)$, all algorithm parameters are set according to their best performance.
Using the same notation in their source papers, we set $\eta=5\times10^{-3} (5.3\times10^{-4}; 5.3\times10^{-5}; 5.2\times10^{-6})$ for GT, $\alpha=1.2\times10^{-2} (1.1\times10^{-3}; 1.2\times10^{-4}; 1.4\times10^{-5})$ and $\beta=0.71 (0.83; 0.83; 0.83)$ for ABm, $\alpha=0.02$, $\tau=0.7$, $\gamma=\frac{4\alpha}{4-4\tau-3\alpha}$, $\eta=0.21 (2\times10^{-2}; 2.1\times10^{-3}; 2.1\times10^{-4})$, $\beta=8.4\times10^{-2} (2\times10^{-2}; 2.1\times10^{-3}; 1.1\times10^{-4})$, and $p=q=0.3$ for OGT. 
Finally, we set $\alpha=0.52 (9.5\times10^{-2}; 2.2\times10^{-2}; 2.2\times10^{-3})$ for DMBFGS.

We can see from Fig.~\ref{fig_2_a} and Fig.~\ref{fig_3_a} that ABm performs best for the problem with a small condition number, i.e., $\kappa_f=10^2$. However, the performance of ABm degrades
 as the condition number increases.
Although OGT has theoretical worst-case optimal complexity, it does not always
have best practical performance for the linear regression test problems. In addition, note that
three rounds of communication per iteration are needed in OGT while other algorithms only require two rounds of communication. Hence, for the problem with $\kappa_f=10^4$, OGT(BEST) is better than NDCG in terms of iteration number but not as good as NDCG  in terms of communication volume.
Finally, we can see from both Fig.~\ref{condition_number} and Fig.~\ref{condition_number1} that 
DMBFGS is in general quite efficient and more robust to the condition number of the objective function compared to other algorithms.

\begin{figure*}[!t]
	\centering
	\subfloat[$\kappa_f=10^2$]{\includegraphics[width=2.5in]{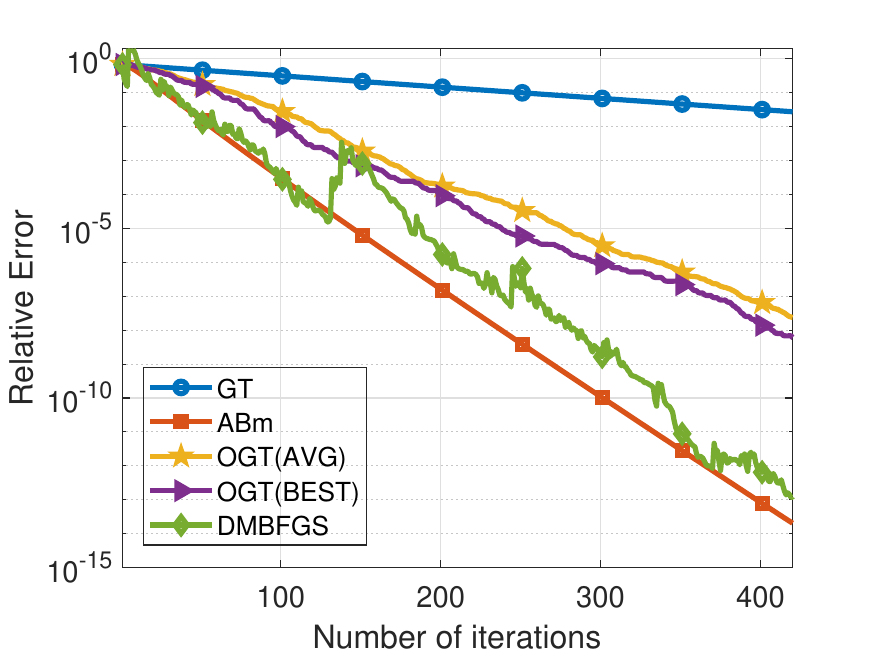}%
		\label{fig_2_a}}
	\subfloat[$\kappa_f=10^3$]{\includegraphics[width=2.5in]{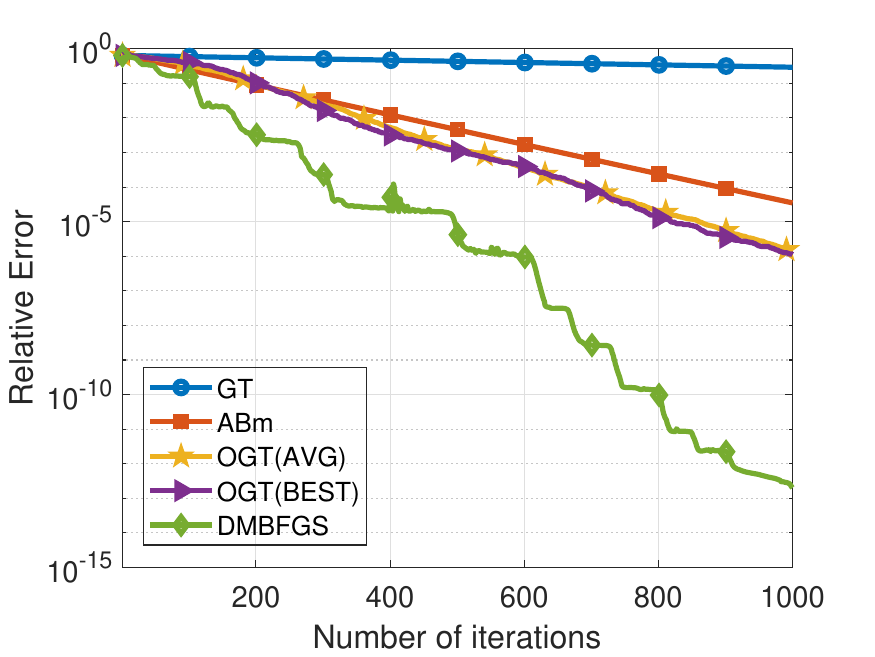}%
		\label{fig_2_b}}
	\hfil
	\subfloat[$\kappa_f=10^4$]{\includegraphics[width=2.5in]{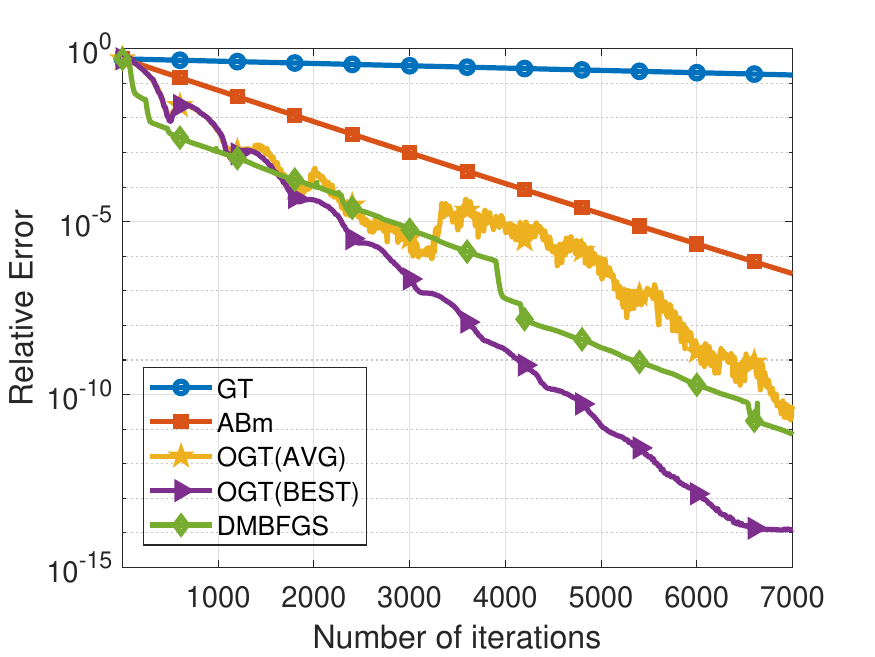}%
		\label{fig_2_c}}
	\subfloat[$\kappa_f=10^5$]{\includegraphics[width=2.5in]{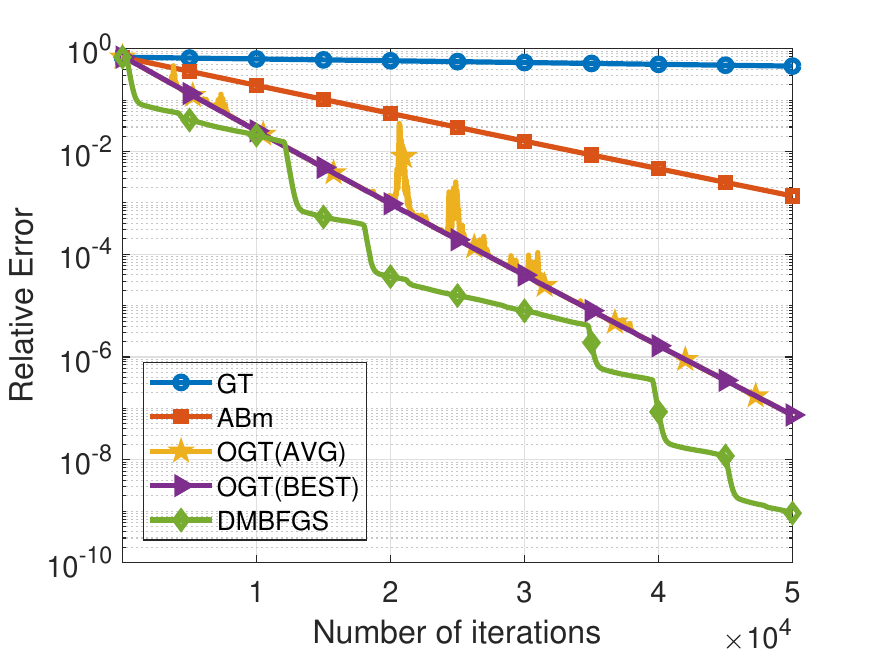}%
		\label{fig_2_d}}
	\caption{Comparisons with gradient-based algorithms for minimizing the strongly convex linear regression problem \eqref{linear_problem} with different condition numbers.}
	\label{condition_number}
\end{figure*}

\begin{figure*}[!t]
	\centering
	\subfloat[$\kappa_f=10^2$]{\includegraphics[width=2.5in]{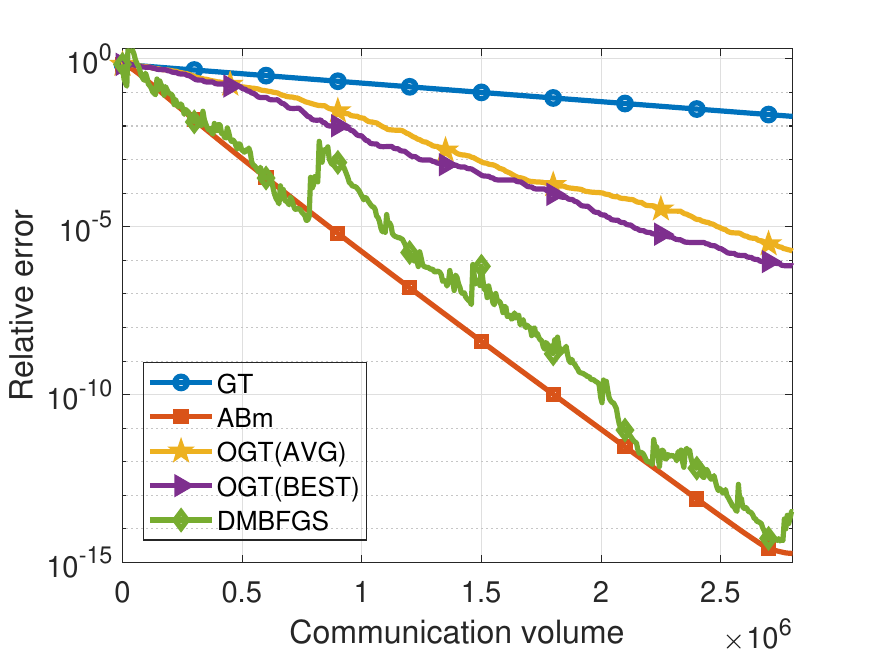}%
		\label{fig_3_a}}
	\subfloat[$\kappa_f=10^3$]{\includegraphics[width=2.5in]{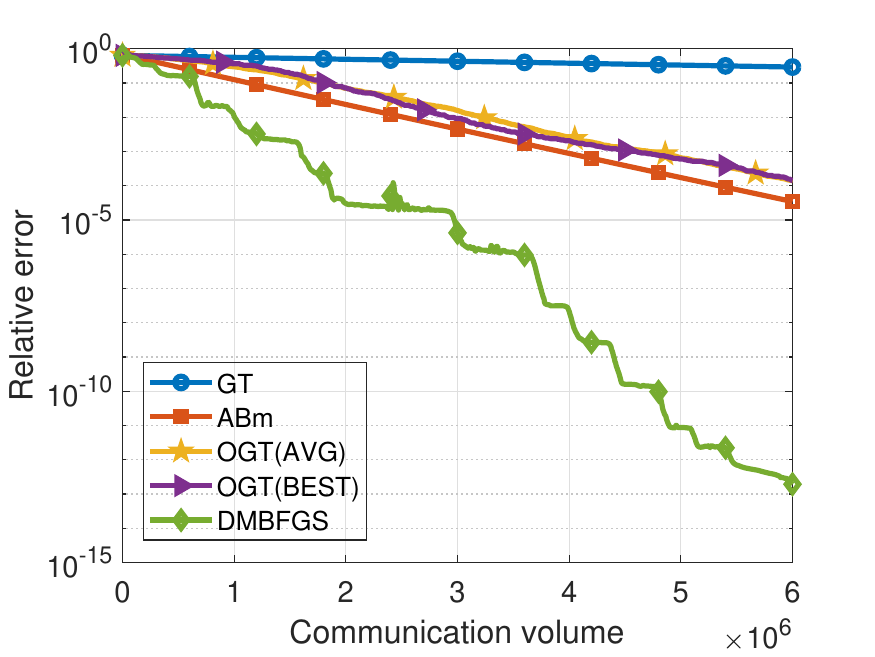}%
		\label{fig_3_b}}
	\hfil
	\subfloat[$\kappa_f=10^4$]{\includegraphics[width=2.5in]{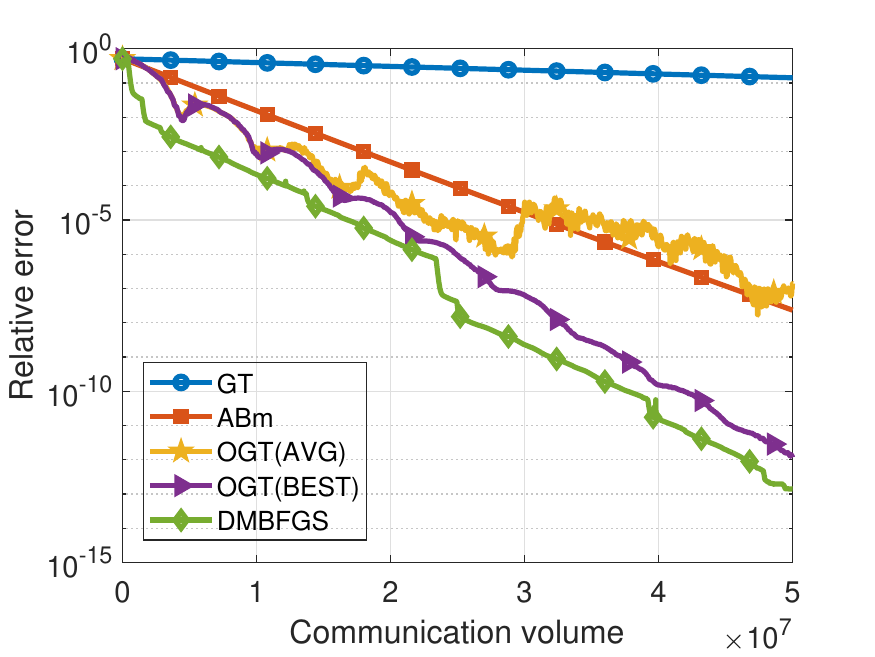}%
		\label{fig_3_c}}	\subfloat[$\kappa_f=10^5$]{\includegraphics[width=2.5in]{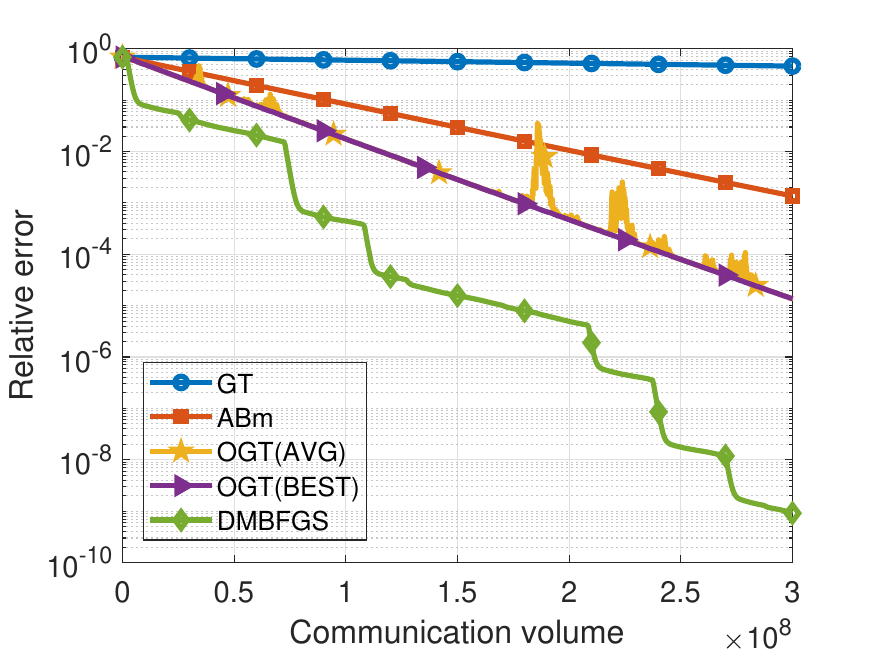}%
		\label{fig_3_d}}
	\caption{Comparisons with gradient-based algorithms for minimizing the strongly convex linear regression problem \eqref{linear_problem} with different condition numbers.}
	\label{condition_number1}
\end{figure*}

\subsubsection{Logistic regression problem}
We consider the logistic regression 
\begin{equation}\label{logistic_problem}
	\mathop {\min }\limits_{\m{z} \in {\R^p}} \sum_{i=1}^n  \sum_{j=1}^{n_i} \log \left(1+\exp (-b_{ij} \m{a}_{ij}\tr \m{z} ) \right)+\frac{ \hat{\lambda}}{2}\|\m{z}\|^2
\end{equation}
with $l_2$-regularization, where $\m{a}_{ij} \in \R^p$ are the feature vectors, $ b_{ij} \in \{-1,+1\}$ are the labels, and $\hat{\lambda}>0$ is the regularization parameter.
The experiments are again conducted on four datasets in Table~\ref{table2}
from the LIBSVM library with $\hat{\lambda}=1$.

We first focus on comparison with gradient-based algorithms ({GT} \cite{qu2017harnessing}, {ABm} \cite{xin2019distributed}, and {OGT} \cite{song2024optimal}). 
Again, all algorithm parameters are set according to their best performance for datasets \textbf{mushroom} (\textbf{ijcnn1}; \textbf{w8a}; \textbf{a9a}).
According to the same notations in their source papers, we set  
$\eta=0.09 (0.12; 0.06; 0.12)$ for GT, $\alpha=0.07 (0.18; 0.07; 0.11)$ and $\beta=0.59 (0.6; 0.58; 0.58)$ for ABm,
 $\alpha=0.2$, $\tau=0.7$, $\gamma=\frac{4\alpha}{4-4\tau-3\alpha}$, $\eta=0.4 (0.6; 0.4; 0.6)$, $\beta=0.2 (0.3; 0.2; 0.3)$, $p=q=0.3$ for OGT and $\alpha=0.18 (0.34; 0.3; 0.32)$ for DMBFGS.

From Fig.~\ref{logistic_iter} and Fig.~\ref{logistic_comm}, we see that DMBFGS is in general more efficient than comparison algorithms in both the iteration number and communication volume
except for the \textbf{mushroom} dataset, where both DMBFGS and ABm perform best and are competitive to each other.
Unlike the linear regression problems, the Hessian of the logistic regression function varies along the iterations.
Thus, we think DMBFGS can more adaptively approximate the Hessian curvature of the objective function and thus yields a better performance.


We now compare DMBFGS with the two well-developed decentralized limited memoryless quasi-Newton algorithms: DR-LM-DFP \cite{zhang2023variance} and D-LM-BFGS \cite{zhang2023variance}.
For datasets \textbf{mushroom} (\textbf{ijcnn1}; \textbf{w8a}; \textbf{a9a}), the algorithm parameters are set for their better performance and use  the same notation in the paper \cite{zhang2023variance}.
In particular, we set $\alpha=0.5 (0.51; 0.52; 0.46)$, $\epsilon=10^{-3}$, $\beta=10^{-3}$, $\C{B}=10^4$, $\tilde{L}=20(20;10;10)$, $M=5 (8; 4; 3)$ for D-LM-BFGS and
 $\alpha=0.04 (0.05; 0.05; 0.05)$, $\rho=0.6 (0.8; 0.8; 0.6)$, $\epsilon=10^{-3}$, $\beta=1$, $\C{B}=10^4$, $\tilde{L}=10(1;1;1)$, $M=6 (5; 5; 5)$ for DR-LM-DFP.

Note that all DR-LM-DFP, D-LM-BFGS and DMBFGS require the same communication cost per iteration. 
We can see from Fig.~\ref{logistic_quasi} that DMBFGS has overall best performance.
We think the reason is that our way of constructing quasi-Newton matrices may utilize more curvature information of the Hessian 
than those obtained by D-LM-BFGS and DR-LM-DFP. In particular, at early iterations, when the average gradient approximation $\m{v}_i^t$ is far from
the exact average gradient $\bar{\m{g}}^t$, the quasi-Newton matrix constructed by  $\check{\m{y}}_i^t = \m{v}_i^{t+1}-\m{v}_i^{t}$ is not a good approximation
of the Hessian and may not even be positive definite. 
In this case, the regularization or damping techniques used by D-LM-BFGS and DR-LM-DFP can not effectively exploit the Hessian curvature information and 
will lead to gradient descent type iterations. 
In contrast, Step 5 of DMBFGS, i.e., Alg.~\ref{alg:Framwork1}, adaptive takes $\hat{\m{y}}_i^t=\m{g}_i^{t+1}-\m{g}_i^{t}$ and $\check{\m{y}}_i^t$ 
to construct quasi-Newton matrix leading to better Hessian curvature approximation.

\begin{figure*}[!t]
	\centering
	\subfloat[\textbf{mushroom}]{\includegraphics[width=2.5in]{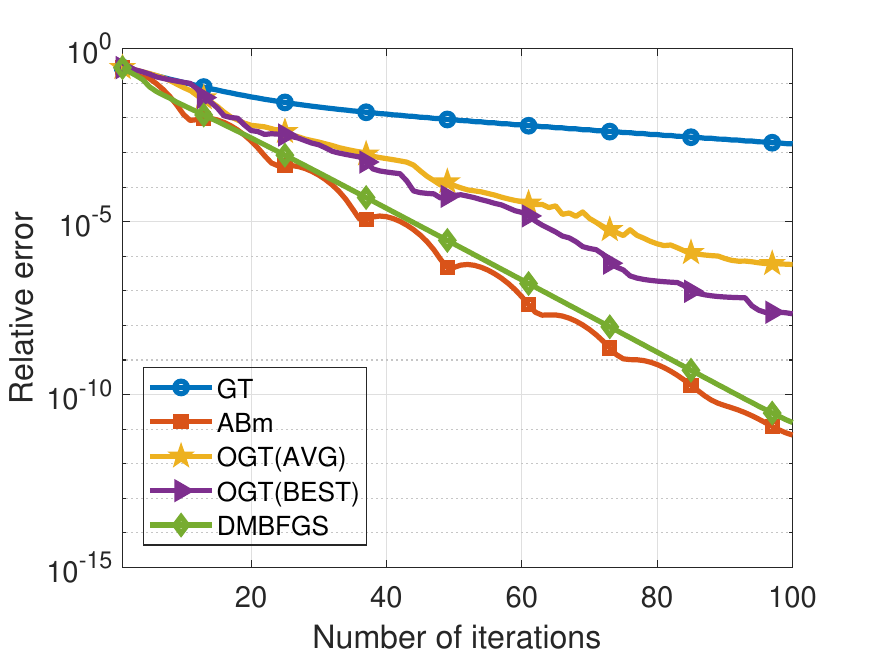}%
		\label{fig_5_a}}
		\subfloat[\textbf{ijcnn1}]{\includegraphics[width=2.5in]{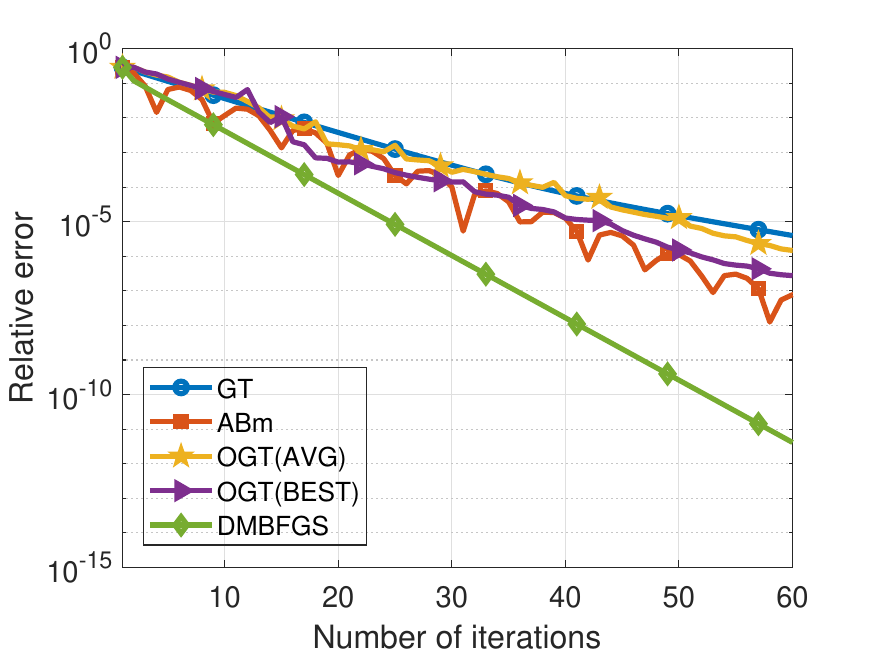}%
		\label{fig_6_a}}
		\hfil
		\subfloat[\textbf{w8a}]{\includegraphics[width=2.5in]{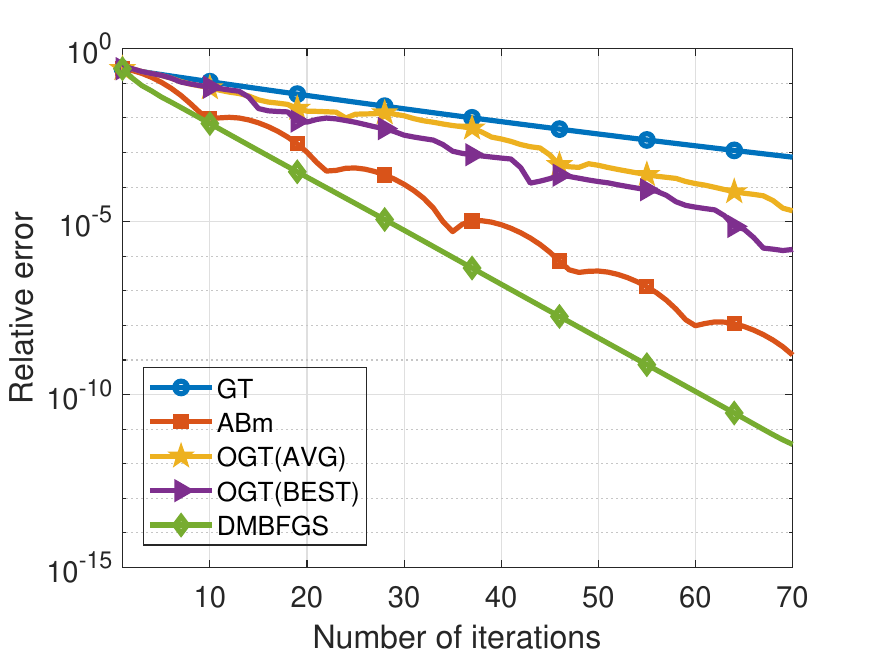}%
		\label{fig_7_a}}
		\subfloat[\textbf{a9a}]{\includegraphics[width=2.5in]{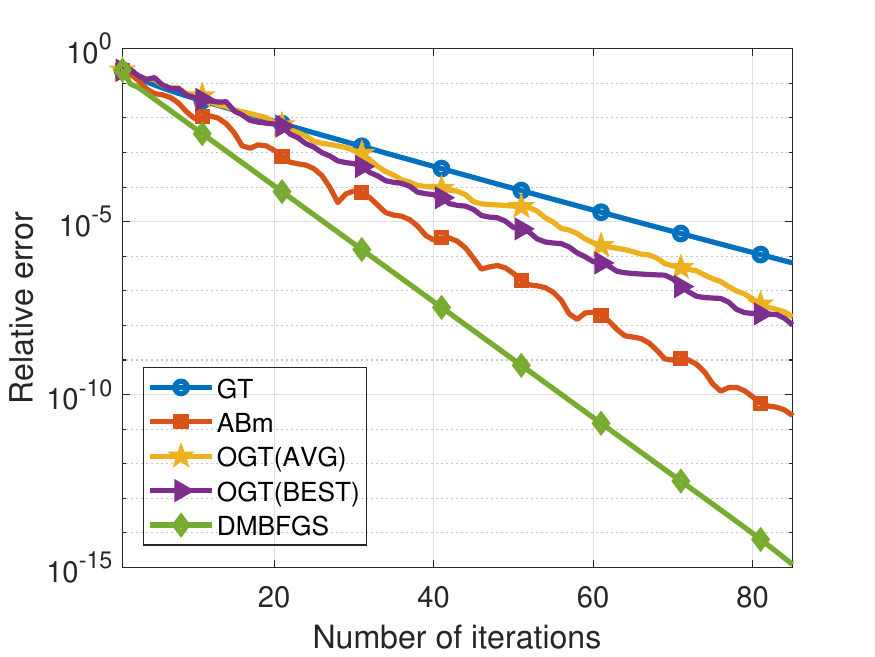}%
		\label{fig_8_a}}
	\caption{Comparisons with gradient-based algorithms for minimizing the strongly convex logistic regression problem \eqref{logistic_problem} on different datasets.}
	\label{logistic_iter}
\end{figure*}

\begin{figure*}[!t]
	\centering
		\subfloat[\textbf{mushroom}]{\includegraphics[width=2.5in]{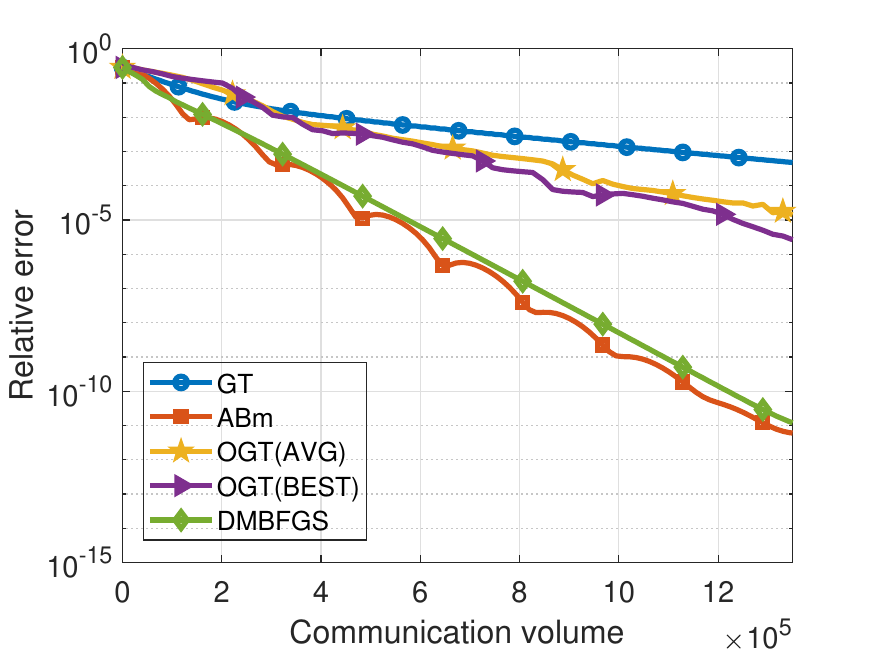}%
		\label{fig_5_b}}
	\subfloat[\textbf{ijcnn1}]{\includegraphics[width=2.5in]{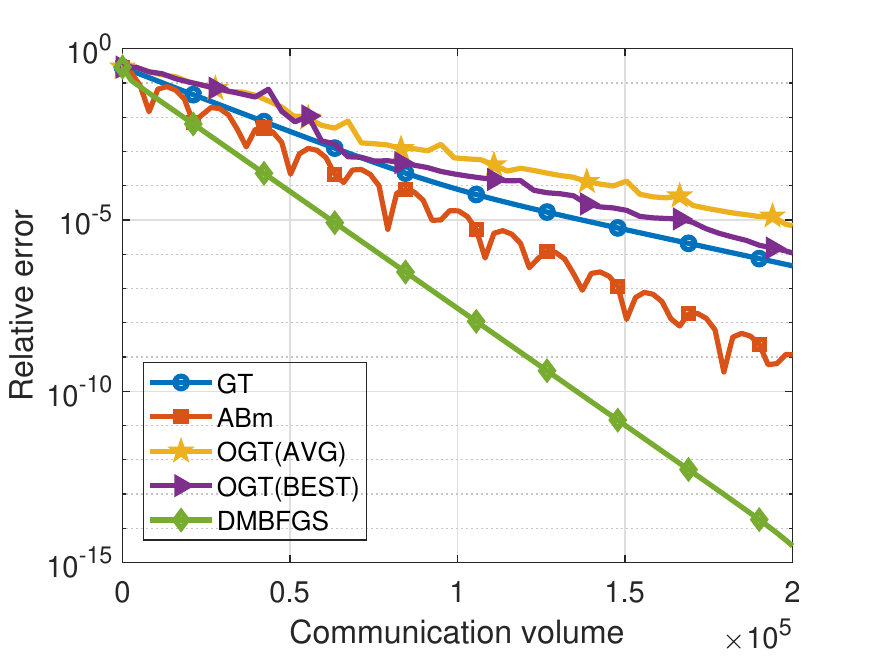}%
		\label{fig_6_b}}
		\hfil
		\subfloat[\textbf{w8a}]{\includegraphics[width=2.5in]{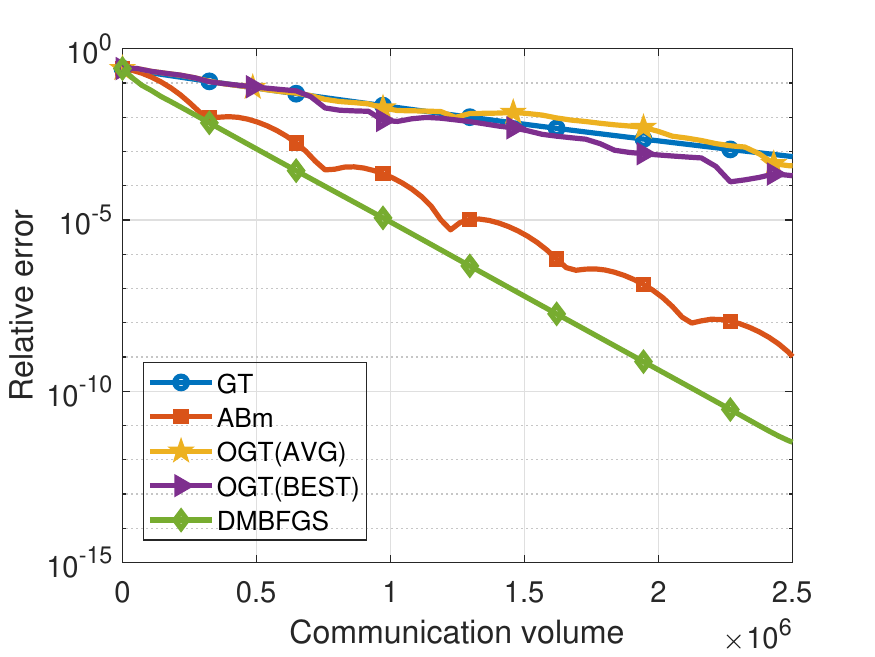}%
		\label{fig_7_b}}
		\subfloat[\textbf{a9a}]{\includegraphics[width=2.5in]{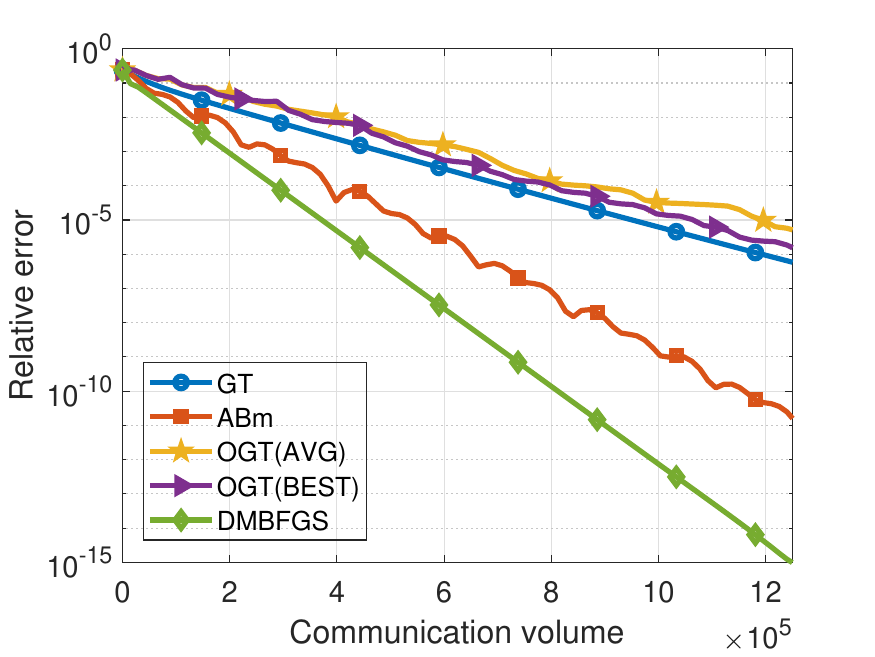}%
		\label{fig_8_b}}
	\caption{Comparisons with gradient-based algorithms for minimizing the strongly convex logistic regression problem \eqref{logistic_problem} on different datasets.}
	\label{logistic_comm}
\end{figure*}

\begin{figure*}[!t]
	\centering
	\subfloat[\textbf{mushroom}]{\includegraphics[width=2.5in]{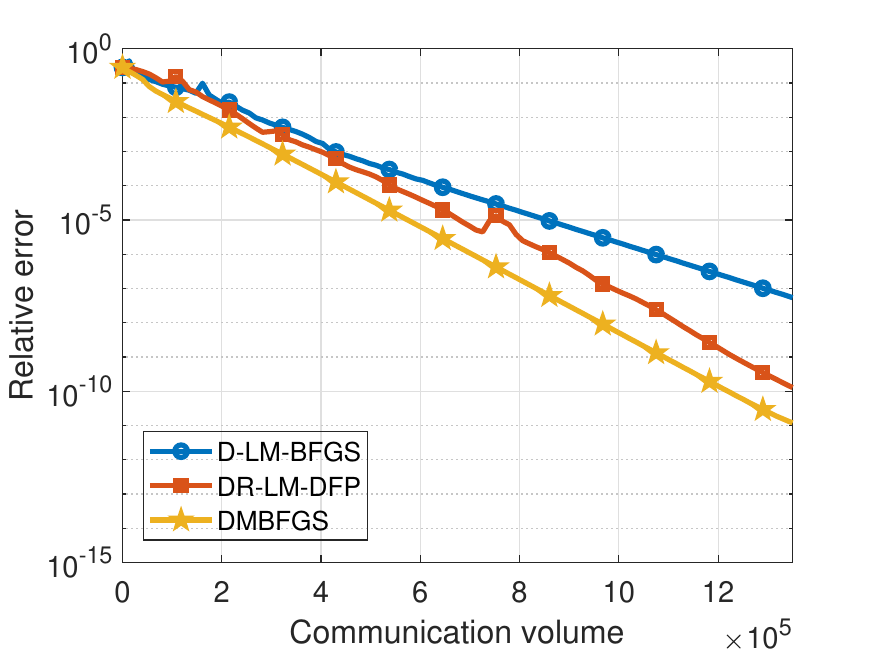}%
		\label{fig_5_c}}
	\subfloat[\textbf{ijcnn1}]{\includegraphics[width=2.5in]{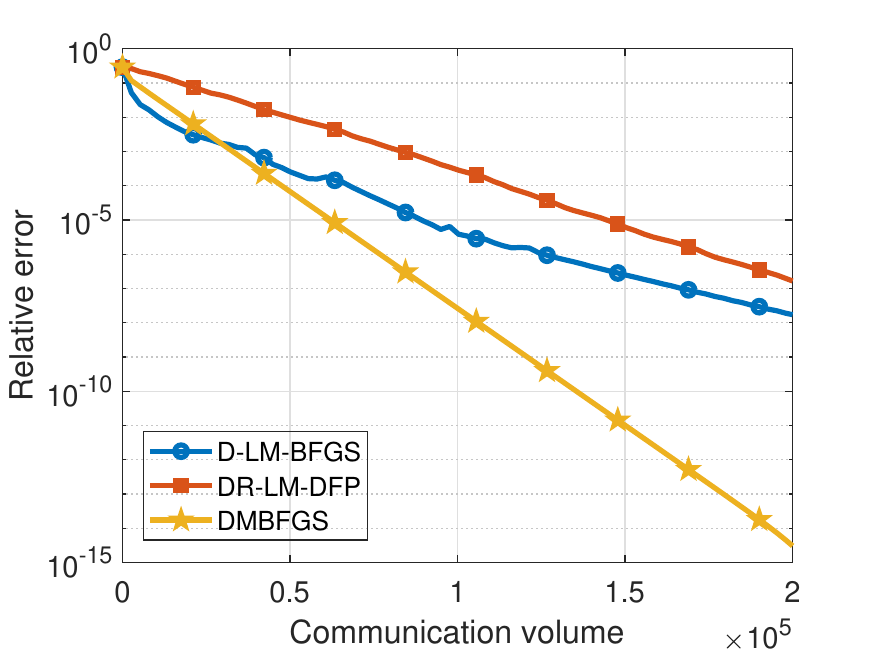}%
		\label{fig_6_c}}
	\hfil
	\subfloat[\textbf{w8a}]{\includegraphics[width=2.5in]{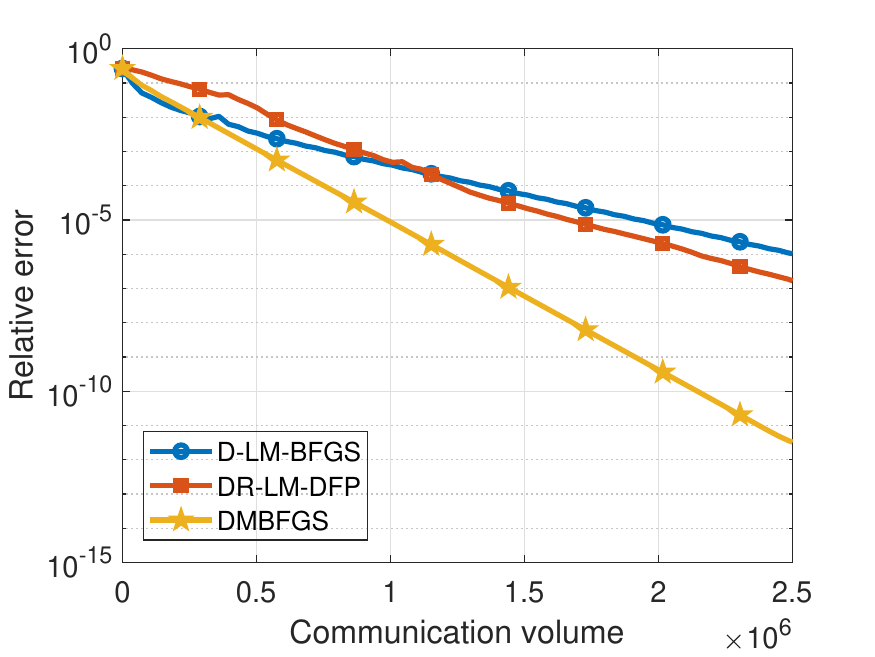}%
		\label{fig_7_c}}
	\subfloat[\textbf{a9a}]{\includegraphics[width=2.5in]{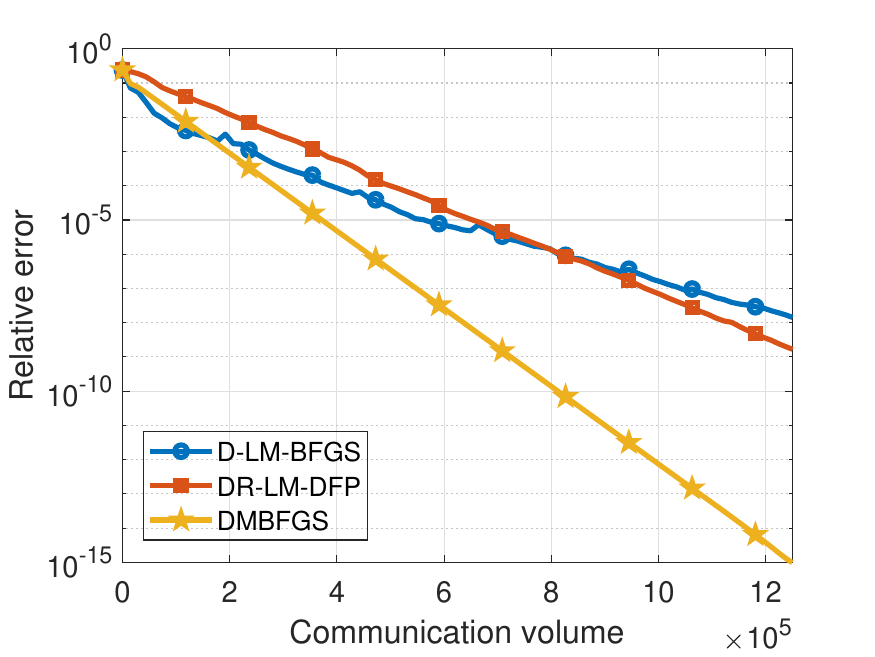}%
		\label{fig_8_c}}
	\caption{Comparisons with limited memoryless quasi-Newton algorithms for minimizing the strongly convex logistic regression problem \eqref{logistic_problem} on different datasets.}
	\label{logistic_quasi}
\end{figure*}

\section{CONCLUSIONS}
This paper considers a decentralized consensus optimization problem and proposes two new algorithms: a new decentralized conjugate gradient (NDCG) method and a decentralized memoryless BFGS method (DMBFGS). 
For nonconvex problems, we propose a new  decentralized conjugate gradient (NDCG) method.
NDCG is proposed for solving \eqref{obj_fun1} when the objective function is nonconvex but Lipschitz continuously differentiable. 
Using the average gradient tracking technique and a newly developed PRP type conjugate parameter. NDCG can take constant stepsizes and has global convergence to 
a stationary point of \eqref{obj_fun1}. 
DMBFGS is targeted for the case when the local objective functions in \eqref{obj_fun1} are strongly convex. 
The memoryless BFGS updating techniques are used in DMBFGS, where the quasi-Newton matrices are constructed by adaptively incorporating local curvature information.
Under proper assumptions, DMBFGS is shown to have globally linear convergence for strongly convex optimization problems. 
Our numerical results on nonconvex logistic regression problems show that NDCG performs better to other gradient-based comparison methods for nonconvex decentralized optimization.
Moreover, our numerical results on solving strongly convex linear regression and logistic regression problems indicate that DMBFGS is more insensitive to the problem condition number
and performs significantly better than both the gradient-based and the quasi-Newton type comparison methods.


\bibliographystyle{IEEEtran}
\bibliography{refs}

\end{document}